\documentclass[12pt]{amsart}
\usepackage{amsmath,bbm}
\usepackage{latexsym,mathrsfs,bm}
\usepackage{color}
\usepackage{url}
\usepackage[english]{babel}
\usepackage{paralist}
\usepackage{amsthm,amsfonts,amssymb,amsmath}
\usepackage[latin1]{inputenc}
\newcommand{\sumtwo}{\operatorname*{\sum\sum}}
\let\originalleft\left
\let\originalright\right
\newcommand{\sumthree}{\operatorname*{\sum\sum\sum}}
\newcommand{\sumfour}{\operatorname*{\sum\sum\sum\sum}}

\newcommand{\new}[1]{\textcolor{black}{#1}}

\newcommand{\eps}{\varepsilon}
\newcommand{\R}{\mathbb{R}}

\newcommand{\N}{\mathbb{N}}

\newcommand{\Z}{\mathbb{Z}}

\newcommand{\es}[1]{\begin{equation}\begin{split}#1\end{split}\end{equation}}
\newcommand{\est}[1]{\begin{equation*}\begin{split}#1\end{split}\end{equation*}}

\newcommand{\M}{\mathcal{M}}

\newcommand{\EE}{\mathcal{E}}

\newcommand{\s}{\operatorname{\mathcal{S}}}
\newcommand{\mathi}{i}
\newcommand{\mathd}{d}
\newcommand{\assign}{:=}

\renewcommand{\mod}[1]{~\pr{\textnormal{mod}~#1}}
\newtheorem*{theo*}{Theorem}
\newtheorem{conjecture}{Conjecture}

\newtheorem{theorem}{Theorem}
\newtheorem{ezer}{Exercise}

\newtheorem{prop}[ezer]{Proposition}

\newtheorem{lemma}{Lemma}

\newtheorem{corollary}{Corollary}

\newtheorem*{rem*}{Remark}
\newcommand{\pr}[1]{\left( #1\right)}

\newcommand{\pmd}[1]{\left| #1\right|}

\newcommand{\e}[1]{\operatorname{e}\pr{ #1}}

\newcommand{\sumprime}{\sideset{}{'}\sum}
\newcommand{\sumstar}{\sideset{}{^*}\sum}

\let\originalleft\left
\let\originalright\right
\renewcommand{\left}{\mathopen{}\mathclose\bgroup\originalleft}
\renewcommand{\right}{\aftergroup\egroup\originalright}

\numberwithin{equation}{section}

\begin{document}
\title[The mean square of the product of $\zeta(s)$ with Dirichlet polynomials]{The mean square of the product of the Riemann zeta function with Dirichlet Polynomials}

\date{\today}

\author[S. Bettin]{Sandro Bettin}
\address{Centre de Recherches Mathematiques Universite de Montreal P. O Box 6128 \\
Centre-Ville Station Montreal \\ Quebec H3C 3J7 
}
\email{bettin@crm.umontreal.ca}

\author[V. Chandee]{Vorrapan Chandee}
\address{Department of Mathematics \\ Burapha University \\ 169 Long-hard bangsaen rd Saensuk, Mueang, Chonburi, Thailand 20131
}
\email{vorrapan@buu.ac.th }

\author[M. Radziwi\l\l]{Maksym Radziwi\l\l}
\address{School of Mathematics \\ Institute for Advanced Study 
\\ 1 Einstein Drive \\ Princeton, NJ, 08540}
\email{maksym@ias.edu}
\curraddr{Centre de Recherches Mathematiques Universite de Montreal P. O Box 6128 \\
Centre-Ville Station Montreal \\ Quebec H3C 3J7 }
\thanks{The third author was partially supported by NSF grant DMS-1128155.}

\allowdisplaybreaks
\numberwithin{equation}{section}
\selectlanguage{english}
\begin{abstract}
Improving earlier work of Balasubramanian, Conrey and Heath-Brown
\cite{BCH}, 
we obtain an asymptotic formula for the mean-square of the
Riemann zeta-function times an arbitrary Dirichlet polynomial of length  
$T^{1/2 + \delta}$, with $\delta = 0.01515\ldots$. 
As
an application we obtain an upper bound of the correct order of magnitude
for the third moment of the Riemann zeta-function. 
We also refine 
previous work of Deshouillers and Iwaniec \cite{DI}, obtaining asymptotic
estimates in place of bounds. 
Using the work of Watt \cite{Watt}, we compute the mean-square of the Riemann
zeta-function times a Dirichlet 
polynomial of length going up to $T^{3/4}$ provided that the Dirichlet 
polynomial assumes
a special shape. Finally, we exhibit a conjectural estimate for trilinear
sums of Kloosterman fractions which implies the Lindel\"of Hypothesis.
\end{abstract}

\maketitle

\section{Introduction}

We are interested in the mean-square of the product of the Riemann
zeta-function $\zeta (s)$ with an arbitrary Dirichlet polynomial $A(s)$. More precisely,
we would like to understand
\begin{equation}\label{eqn:defofI}
 I = \int_{\mathbbm{R}} \left| \zeta \left( \tfrac{1}{2} + \mathi t \right)
   \right|^2 \cdot \left| A \left( \tfrac{1}{2} + \mathi t \right) \right|^2
   \phi \left( \frac{t}{T} \right) \mathd t 
\end{equation}
with $\phi (x)$ a smooth function supported in $[1, 2]$ and
\[ A (s) \assign \sum_{n \leqslant T^{\theta}} \frac{a_n}{n^s} ,\quad a_n \ll n^{\varepsilon}, \quad \theta < 1. \]

Asymptotic estimates for $I$
have been used consistently to understand the distribution of
values of $L$-functions, the location of their zeros, and upper and lower bounds
for the size of $L$-functions. See, for example,~\cite{CGG, Con, Rad2, S}.

It is crucially important 
to allow $\theta$ to be as large
as possible. For example, if we could take $\theta = 1 - \varepsilon$
in~\eqref{eqn:defofI} then the
Lindel\"{o}f Hypothesis would follow.

Balasubramanian, Conrey and Heath-Brown obtained an asymptotic formula for
$I$ when $\theta < \tfrac{1}{2}$. 
For $\theta < \tfrac{1}{2}$ and
$\phi (t)$ the indicator function of the interval $[1, 2]$, they show that
\begin{equation}\label{asfbch}
  I = T \sum_{d, e \leqslant T^{\theta}} \frac{a_d \overline{a_e}}{[d, e]} \cdot
  \left( \log \left( \frac{T (d, e)^2}{2 \pi de} \right) + 2 \gamma + \log 4 - 1
  \right)+o(T).
\end{equation}
When $A (s)$ is a mollifier, they show 
that one can go further and take $\theta < \tfrac 12 
+ \tfrac {1}{34} = 0.529411\ldots$ . Their motivation
was to understand the location of the zeros of the Riemann zeta-function.
Specifically they deduce that at least $38$\% of the complex zeros of $\zeta
(s)$ are on the critical line $\Re s = \tfrac{1}{2}$. Improvements on the
admissible length of $A (s)$ will lead to a further understanding of the zeros
of $\zeta (s)$ on the critical line. (See also \cite{Con}).

In complete generality the formula~\eqref{asfbch} fails when $\theta > 1$.
Balasubramanian, Conrey and Heath-Brown conjecture that it remains true
provided that $\theta < 1$. This is known as the $\theta = 1$ conjecture. An
important change occurs at $\theta = \tfrac{1}{2}$ . When $\theta <
\tfrac{1}{2}$ only the diagonal terms (in the sense of Section~\ref{sec:diagonalterm} below) contribute to $I$, while for $\theta >
\tfrac{1}{2}$ there is also a contribution from the non-diagonal terms which
seems difficult to manage given the generality of the Dirichlet polynomial $A
(s)$. The main result of our paper consists in breaking the $\tfrac{1}{2}$ barrier 
for an arbitrary Dirichlet polynomial. In fact, we prove~\eqref{asfbch} for 
$\theta < \frac{17}{33} = \tfrac{1}{2} + \delta$ with $\delta = \frac{1}{66} \approx 0.01515...$.

\begin{theorem} \label{thm:asymptotic}
  Let $I$ and $A (s)$ be as above. If $\theta < \tfrac{1}{2} + \delta$, with
  $\delta = \frac{1}{66}$ then,
  \est{ I &=  \sum_{d, e \leqslant T^{\theta}} \frac{a_d
     \overline{a_e}}{[d, e]} \cdot \int_\R \left( \log \left( \frac{t\, (d, e)^2}{2
     \pi de} \right) + 2 \gamma\right)\,{\phi} \pr{\frac tT} \mathd t + O \left( T^{\frac 3{20} + \varepsilon} N^{\frac{33}{20}} +T^{\frac13+\eps}\right) , }
where $N := T^{\theta}$. 
\end{theorem}

We notice that the off-diagonal terms contribute to the main term 
roughly those $d$ and $e$
for which the logarithm in the above expression is negative.

Our main tool in the proof of Theorem~\ref{thm:asymptotic} 
is an estimate for trilinear forms of
Kloosterman fractions, which will appear in~{\cite{BCR}}. 
This estimate improves a result of Duke, Friedlander, Iwaniec in
{\cite{DFI1}}, dealing with bilinear sums. The use of Theorem 2 in their paper 
is also enough to break the $\frac12$ barrier, though with 
the smaller constant $\delta = 1/190 \approx 0.00526$ in Theorem~\ref{thm:asymptotic}. 

If we assume a general estimate for trilinear forms of Kloosterman fractions, such as,
\es{\label{eqn:gentrilinear}
S_{A,M,N} &:= \sum_a  \sumtwo_{(m,n) = 1} \nu_a\alpha_m \beta_n \e{\frac{a\overline{m}}{n}} \\
&\ll_{\eps} \| \alpha\|\|\beta\|\|\nu\|(M+N)^{\frac 12 + r + \eps}A^{t} +\|\nu\|A^\frac12\pr{ \|\alpha\|_\infty\| \beta\| N^{\frac 12 + \eps}+ \| \alpha\| \|\beta\|_{\infty}M^{\frac 12 + \eps}},
} 
where $M \leq m < 2M,$ $N \leq n < 2N$, $A \leq a < 2A$,  $A \ll (NM)^{\frac{0.5-r}{1 + 2t} + \eps}$, and $\|\cdot\|$ and $\|\cdot\|_\infty$ denote the $L_2$ and $L_\infty$ norms respectively, then the statement of Theorem~\ref{thm:asymptotic} can be replaced as follows.

\begin{theorem} \label{thm:asymwithgeneraltrilinear}
Suppose that~\eqref{eqn:gentrilinear} is true for some $r,t\geq 0$. Then
\est{ I =  \sum_{d, e \leqslant T^{\theta}} \frac{a_d
     \overline{a_e}}{[d, e]} \cdot&\int_\R\left( \log \left( \frac{t \,(d, e)^2}{2
     \pi de} \right) + 2 \gamma \right) \phi \pr{\frac tT} \mathd t
      + O \left(T^{\frac 12 - t + \eps} N^{\frac 12 + r + 2t} +T^{\frac13+\eps} \right), }
for $\theta < \tfrac{1}{2} + \tfrac{0.5 - r}{1 + 2(r + 2t)}$
and where $N := T^{\theta}$. 
\end{theorem}
The estimate of Duke, Friedlander, Iwaniec implies~\eqref{eqn:gentrilinear} with
$r = \tfrac{23}{48}$ and $t = \tfrac 12$, while
the estimate of Bettin and Chandee allows us to take 
$r = \tfrac {9}{20}$ and $t = \tfrac{7}{20}$.
%$We expect~\eqref{eqn:gentilinear} to be false for $r < 0$%
We conjecture that~\eqref{eqn:gentrilinear} holds true for all $r,t\geq0$. 
\begin{conjecture}\label{mconj} Let $A\ll (NM)^{\frac12+\eps}$. Then
\begin{multline} \label{trilinear} 
S_{A,M,N}  \ll \| \alpha\|\|\beta\|\|\nu\|(M+N)^{\frac 12 + \eps} +\|\nu\|A^\frac12\pr{ \|\alpha\|_\infty\| \beta\| N^{\frac 12 + \eps}+ \| \alpha\| \|\beta\|_{\infty}M^{\frac 12 + \eps}}.
\end{multline}
\end{conjecture}

This conjecture essentially states that we expect square-root cancellation in the shortest two sums, as long as the total saving does not exceed $M$ or $N$. In the Appendix we show that this is best possible, up to $\varepsilon$-powers.  

Using the estimate~\eqref{trilinear} and Theorem~\ref{thm:asymwithgeneraltrilinear}, we obtain an asymptotic formula for $I$ valid for any $\theta < 1$, and this implies the Lindel\"{o}f hypothesis. We state this as a corollary below.
\begin{corollary}\label{cor:lindelof}
Suppose that Conjecture~\ref{mconj} holds. Then the Lindel\"of Hypothesis is true.
\end{corollary}
Conjecture \ref{mconj} appears to be strictly stronger than 
the Lindel\"of Hypothesis. Indeed Conjecture 1 implies (\ref{asfbch}) with $\theta < 1$, 
while the Lindel\"of Hypothesis only
gives the cruder bound
$$
I \ll T^{1 + \varepsilon} \sum_{n \leq T^{\theta}} \frac{|a_n|^2}{n}.
$$

The proof of Theorem~\ref{thm:asymwithgeneraltrilinear}, on which Corollary \ref{cor:lindelof} depends, 
is the same as that of Theorem~\ref{thm:asymptotic} except that we use~\eqref{eqn:gentrilinear} instead of Proposition~\ref{prop:DFIimproved}. The modification will be discussed at the end of Section~\ref{sec:Thm1}. 

Duke, Friedlander and Iwaniec apply their estimate to obtain bounds for the
twisted second moment of a Dirichlet $L$-function \cite{DFI2}. They show that,
\[ \sum_{\chi \text{ mod } q} \left| L \left( \tfrac{1}{2}, \chi \right)
   \right|^2 \cdot \left| D \left( \tfrac{1}{2}, \chi \right) \right|^2 \ll
   q^{1 + \varepsilon} \]
for Dirichlet polynomials $D (s, \chi)$ with coefficients $a_n \ll
n^{\varepsilon}$ and of length $q^{1 / 2 + \delta'}$ with some $\delta' > 0$.
Our proof of Theorem~\ref{thm:asymptotic} would not extend to give an asymptotic
formula in this case, and additional input 
is needed.

As an application of Theorem~\ref{thm:asymptotic} we obtain an upper bound of the correct order
of magnitude for the third moment of the Riemann zeta-function.

\begin{corollary}\label{cor:thirdMoment}
We have,
\[ \int_T^{2 T} \left| \zeta \left( \tfrac{1}{2} + \mathi t \right) \right|^3
   \mathd t \ll T (\log T)^{9 / 4} . \]
\end{corollary}

We further indicate in Section~\ref{sec:overnMoment} how to refine this result to obtain correct
upper bounds for the $2k$-th moment, when $k$ has the form $k = 1 + 1 / n$. Previously Corollary~\ref{cor:thirdMoment}
was known only on the assumption of the Riemann Hypothesis \cite{HB}. The only sharp unconditional upper bounds that were previously known are for the classic cases $k = 0,1,2$ and for $k = 1 / n$, due to Heath-Brown \cite{HB}.

With further applications in mind we investigate how much $\theta$ can be
increased when the Dirichlet polynomial $A (s)$ is specialized.
% to a special
%form.

\subsection{Products of two Dirichlet polynomials.}

When $A (s)$ can be written as a product of two Dirichlet polynomials $B (s) C
(s)$, one can appeal to stronger estimates for sums of Kloosterman sums due to Deshouillers
and Iwaniec.
In {\cite{DI}}, Deshouillers and Iwaniec consider the product
of $\zeta (s)$ with two Dirichlet polynomials,
\begin{equation}\label{eqn:defofJ}
 J = \int_{\mathbbm{R}} \left| \zeta \left( \tfrac{1}{2} + \mathi t \right)
   \right|^2 \cdot \left| A \left( \tfrac{1}{2} + \mathi t \right) \right|^2
   \cdot \left| B \left( \tfrac{1}{2} + \mathi t \right) \right|^2 \mathd t
\end{equation}
with 
\es{\label{afr}
A(s)\assign \sum_{n \leqslant N} 
\frac{\alpha_n}{n^s} ,\quad B (s) \assign \sum_{k \leqslant K} 
\frac{\beta_k}{k^s},\quad \text{where  }  \alpha_n\ll n^\eps,\,\beta_k \ll k^{\varepsilon}.
}
They show that
if $N \geq K$, then
\[ J \ll T^{\varepsilon} \cdot (T + T^{1 / 2} N^{3 / 4} K + T^{1 / 2} NK^{1 /
   2} + N^{7 / 4} K^{3 / 2}) . \]
Their proof depends on estimates for incomplete Kloosterman sums as developed
in {\cite{DI}}. Proceeding similarly as in the proof of
Theorem~\ref{thm:asymptotic}, and using Deshouillers and Iwaniec's estimate,
we refine their 
bound to an asymptotic estimate.

\begin{theorem} \label{thm:zetatimestwopoly}
  Let $J, A (s)$ and $B (s)$ be as defined in~\eqref{eqn:defofJ} and~\eqref{afr}, and let $N \geqslant K$.
  Then,
\est{
    J & =  \sum_{d, e \leqslant NK} \frac{a_d \overline{a_e
    }}{[d, e]} \cdot\int_\R \left( \log \left( \frac{t (d, e)^2}{2 \pi de} \right)
    + 2 \gamma \right)\phi\pr{\frac tT}\mathd t \\
    &\quad + O (T^{\varepsilon} \cdot (T^{1 / 2} N^{3 / 4} K + T^{1 / 2} NK^{1
    / 2} + N^{7 / 4} K^{3 / 2})),
}
with $a_d \assign \sum_{nk = d} \alpha_n \beta_k$.
\end{theorem}

When the length of $N$ and $K$ is chosen suitably, Theorem~\ref{thm:zetatimestwopoly} allows
us to take $\theta < \tfrac{1}{2} + \frac{1}{10}$. 

\subsection{Specializing one of the Dirichlet polynomials}

A specific case of
interest is $A(s)B(s)$
with  $A(s)$ of length  $N = \sqrt{T}$ and smooth coefficients, and $B(s)$ arbitrary and as long
as possible. One can think of such estimates as estimates for 
the twisted fourth moment of the Riemann zeta-function. 
In this case we can go further by combining the trilinear sums estimate used to prove Theorem~\ref{thm:asymptotic} with Watt's strengthening~\cite{Watt} of the groundbreaking work of Deshouillers-Iwaniec on estimates for sums of Kloosterman sums {\cite{DI2}}. 

\begin{theorem}\label{thm:zetatimesproductoftwosmooth}
  Let $J, A (s)$ and $B (s)$ be as defined in~\eqref{eqn:defofJ} and~\eqref{afr}. Let $N \ll T^{\frac12+\eps}$
  for all $\eps>0$ and assume that $\alpha_n = \psi (n)$ with $\psi (x)$ a smooth
  function such that $\psi^{(j)} (x) \ll_j x^{-
  j}$ for all $j > 0$. Let $K\ll T^{\frac14}$ and $\beta_k \ll k^{\varepsilon}$ for all $\eps>0$. Moreover assume $\alpha_n$ is supported on $[N T^{-\xi_1}, 2N]$ and $\beta_k$ is supported on $[K T^{-\xi_2}, 2K]$, where $0\leq \xi_1\leq \frac15,$ $0\leq \xi_2\leq \frac1{16}$.  Then,
\est{    J & =  \sum_{d, e \leqslant NK} \frac{a_d 
      \overline{a_e
    }}{[d, e]} \cdot \int_\R \left( \log \left( \frac{t (d, e)^2}{2 \pi de} \right)
    + 2 \gamma \right)\phi\pr{\frac tT} \mathd t  +{}\\
    &\quad  + O \left( T^{\tfrac{1}{2} + \varepsilon } K^2 + KN^\frac34T^{\frac 38 + \eps} + \ T^{\frac{39}{40}+ \frac18\xi_1 + \frac{2}{5}\xi_2+ \eps}\right) ,
}
  where $a_d = \sum_{nk = d} \alpha_n \beta_k$.
\end{theorem}

\begin{rem*}
Theorem~\ref{thm:zetatimesproductoftwosmooth} yields an asymptotic formula for $5\xi_1 + 16\xi_2<1$ (and $N\ll T^\frac12$, $K\ll T^{\frac12-\eps}$). We remark that this range could be enlarged with a little more work.
\end{rem*}

We notice that Theorem~\ref{thm:zetatimesproductoftwosmooth} allows us to take $\theta < \tfrac{3}{4}$ for
Dirichlet polynomials of the form $A (s) B (s)$ with $A (s)$ pretending to be
$\zeta (s)$ and $B (s)$ of length up to $T^{1 / 4 - \varepsilon}$. Thus, following the work of Radziwi\l\l \ \cite{Rad}, Theorem~\ref{thm:zetatimesproductoftwosmooth} could be applied to give a sharp upper bound for the $2k$-th moment of the Riemann zeta function for $2k<5$, conditionally on the Riemann hypothesis (however, we remark that this has been recently proven for all $k\geq0$ by Harper~\cite{Harper}). It would be interesting to investigate if Theorem~\ref{thm:zetatimesproductoftwosmooth} has other applications, for example to the study of large gaps between the zeros of the Riemann
zeta-function (see \cite{Bredberg}). 

Theorem~\ref{thm:zetatimesproductoftwosmooth} refines upon Watt's result, who uses his Kloosterman sum estimate to give (essentially) an upper bound of the form $J\ll T^{1+\eps} + T^{1/2 + \varepsilon} K^2$, for $a_n,b_n$ supported on dyadic intervals. Theorem~\ref{thm:zetatimesproductoftwosmooth} should also be compared with the asymptotic formula for the twisted fourth moment of Hughes and Young \cite{HY}. Their result allows to get an asymptotic formula for the second moment of $\zeta^2 (s) B(s)$ with $B(s)$ of length up to $T^{1 / 11 - \varepsilon}$.

\section*{Acknowledgments}

We are very grateful to Brian Conrey for suggesting to us the problem of breaking the $\tfrac12$ barrier in Theorem~\ref{thm:asymptotic} and to Micah B. Milinovich and Nathan Ng for pointing out the paper of Duke, Friedlander, Iwaniec~{\cite{DFI1}}.
We also wish to thank the referee for a very careful reading of the paper and for indicating several inaccuracies and mistakes. 

\section{Estimates for sums of Kloosterman sums  }
\begin{rem*}
Throughout the paper, we use the common convention in analytic number theory that $\eps$ denotes an arbitrarily small positive quantity that may vary from line to line.
\end{rem*}

In this section, we collect the estimates for sums of Kloosterman sums that will be used to prove the theorems. 

The following Proposition is from \cite{BCR}, and we will use it when dealing with the contribution of the off-diagonal terms in Theorem~\ref{thm:asymptotic} and~\ref{thm:zetatimesproductoftwosmooth}. 

\begin{prop}\label{prop:DFIimproved}
Let $\alpha_m$, $\beta_n$, $\nu_a$ be complex numbers, where $M \leq m < 2M,$ $N \leq n < 2N$, and $A \leq a < 2A$.  Then for any $\eps > 0$, we have
\es{\label{mpb}
\sum_a  \sumtwo_{(m,n) = 1} \nu_a\alpha_m \beta_n \e{\frac{a\overline{m}}{n}}&\ll_\eps \|\alpha\| \|\beta\|\|\nu\| \Big(1+\frac{A}{MN}\Big)^\frac{1}{2}\\
&\quad\times\hspace{-0.25em}\pr{(AMN)^{\frac7{20}+\eps}(M+N)^{\frac14}+(AMN)^{\frac38+\eps}(AN+AM)^\frac18},
}
where $\|\cdot\|$ denotes the $L_2$ norm.
\end{prop}

The off-diagonal terms in Theorem~\ref{thm:zetatimestwopoly} will be estimated using the following bound, due to Deshouillers and  Iwaniec~\cite{DI}.

\begin{prop}[Deshouillers, Iwaniec] \label{lem:DI}
Let $L, J, U, V \geq 1$ and $|c(u, v)| \leq 1.$ We then have
\est{&\sum_{1 \leq \ell \leq L} \sum_{\substack{1 \leq j \leq J \\ (\ell, \varrho j) = 1}} \left| \sum_{1 \leq u \leq U} \sum_{\substack{1 \leq v \leq V \\ (v, \ell) = 1}}  c(u, v) \e{u\frac{\overline{\varrho v j} }{\ell}}\right|  \\
&\ll (LJUV)^{1/2 + \eps} \left\{ (LJ)^{1/2} + (U + V)^{1/4}[LJ(U + \varrho V)(L + \varrho V^2) + \varrho UV^2 J^2]^{1/4}\right\}.}
\end{prop}

Finally, to estimate the off-diagonal terms in Theorem~\ref{thm:zetatimesproductoftwosmooth}, we will use the following Proposition, which can be derived easily from Proposition 4.1 of Watt~\cite{Watt}.
\begin{prop}[Watt]\label{Watt}
Let $H,C,P,V,R,S\geq1$ and $\delta\leq1$. Assume that for some $\eps>0$ we have 
\es{\label{cond}
X&:=\pr{\frac {RVSP}{HC}}^\frac12\gg (RSPV)^{\eps},\\
(RS)^2&\geq\max\left(H^2C,\frac{SP}{V}(RSPV)^{\eps}\right).
}
Moreover, assume that $\alpha(x), \beta(x)$ are complex valued smooth functions, supported on the intervals $[1/2,H]$ and $[1/2,C]$ respectively, such that
\est{
\alpha^{(j)}(x),\beta^{(j)}(x)\ll_j (\delta x)^{-j}
}
for all $j\geq0$. Assume $a_r,b_s$ are sequences of complex numbers supported on $[R/2,R]$, $[S/2,S]$ respectively and are such that $a_r\ll r^\eps$, $b_s\ll s^\eps$. Finally, assume that for all $i,j\geq 0$,
\est{
\frac{d^{i+j}}{dy^idy^j}\gamma_{r,s}(x,y)\ll_{i,j} x^{-i}y^{-j},
}
where for all $r$ and $s$, $\gamma_{r,s}(x,y)$ is supported on $[V/2,V]\times[P/2,P]$.
Then 
\es{\label{oosi}
&\sumtwo_{s\sim S,\ r\sim R}\sumfour_{\substack{h,c,p,v,\\(rv,sp)=1}}\alpha(h)\beta(c)\gamma_{r,s}(v,p)a_{r}b_{s}\e{\pm\frac{hc\overline {rv}}{sp}}\\
&\quad\ll \delta^{-\frac 72} HC  R(V+SX)\pr{1+\frac{HC}{RS}}^\frac12\pr{1+\frac{P}{VR}  }^\frac12\pr{1+\frac{H^2CPX^2}{VS^3R^4}}^\frac14 (HCRVPS)^{20\eps}.
}
\end{prop}
\begin{proof}
Firstly using smooth partitions of unity, we can assume that $\alpha(x), \beta(x)$ are supported on $[H/2,H]$ and $[C/2,C]$, since the bound~\eqref{oosi} is weaker (and the conditions~\eqref{cond} stricter) for larger values of $H$ and $C$. Moreover, by dividing by $R^\eps$ and $S^\eps$ if necessary, we can assume $a_r,b_s\ll 1$.

By Poisson's formula,
\est{
\sum_{(v,sp)=1}\gamma(v,p)\e{\pm\frac{hc\overline {rv}}{sp}}&=\sumstar_{u\mod {sp}}\e{\pm\frac{hc\overline {ru}}{sp}}\sum_{v\equiv u\mod {sp}}\gamma(v,p)\\
&=\sumstar_{u\mod {sp}}\e{\pm\frac{hc\overline {ru}}{sp}}\frac1{sp}\sum_{\ell}\e{-\frac{\ell u}{sp}}\int_{\R}\gamma(y,p)\e{\frac {\ell y}{sp}}\,dy\\
&=\sum_{\ell}S(hc\overline r,\mp\ell ,sp)\int_{\R}\gamma(ysp,p)\e{\ell y}\,dy.
}
If $\ell=0$, the Kloosterman sum reduces to a Ramanujan sum, and one has $S\pr{hc\overline {r},\mp \ell ,sp}\ll (hc,sp)$. Thus, the contribution to~\eqref{oosi} coming from the terms $\ell=0$ is bounded by
\est{
\sumtwo_{s\sim S,\ r\sim R}\sumthree_{\substack{h,c,p \\(r,sp)=1}}\alpha(h)\beta(c)a_{r}b_{s}\frac{(hc,sp)}{sp}V\ll R VHC (HCP)^{\eps}.
}
Also, integrating by parts repeatedly, we see that the terms with $\ell \geq\frac{ SP}{ V}(RSPV)^{\eps}$ give a negligible contribution. For the remaining terms, we introduce a smooth partition of unity 
\est{
1=\sumprime_{L}\theta_L(x),\qquad \forall x\geq1,
}
where $\theta_L(x)$ is supported in $[L/2, 3L]$ (with $L\ll \frac{ SP}{ V}(RSPV)^{\eps}$), satisfies $\theta_L(x)^{j}\ll_j L^{-j}$ for all $j\geq0$, and is such that $\sum'_{L\leq X}1\ll \log (2+X)$ for all $X\geq1$. Thus, we need to bound
\est{
&\sumtwo_{s\sim S,\ r\sim R}\sumthree_{\substack{h,c,p,\\(r,sp)=1}}\sum_{0< |\ell|<\frac{SP}{V}(RSPV)^{\eps}}\alpha(h)\beta(c)a_{r}b_{s}S(hc\overline r,\mp\ell ,sp)\int_{\R}\gamma(ysp,p)\e{\ell y}\,dy\\
&=\sumprime_{L}\int_{y\sim \frac{V}{SP}} \sumtwo_{s\sim S,r\sim R}\sumthree_{\substack{h,c,p,\\(r,sp)=1}}\sum_{\ell\sim L}\alpha(h)\beta(c)\omega(\ell,y)a_{r}b_{s}S(hc\overline r,\mp\ell ,sp)f_{s}(p,y)\,dy,\\
}
where $f_{s}(p,y):=\gamma(ysp,p)$, $\omega(\ell, y)=\theta(\ell)\e{\ell y}$, and $\sum^\prime$ denotes the sum over the partitions of unity. We remark that for $y\sim \frac{V}{SP}$, we have $\frac{d^j}{dp^j}f_{s}(p,y)\ll  p^{-j}$, and that $\frac{d^j}{d\ell^j}\omega(\ell, y)\ll ( L^{-1}+\frac{V}{SP})^{j}\ll L^{-j}(RSPV)^{j\eps}$. By Proposition 4.1 of Watt~\cite{Watt}, the sums inside the integral are bounded by
\est{
\delta^{-\frac72}(RVLP)^{\frac72\eps}HCL(RSX)^{1+\eps}\pr{1+\frac{HC}{RS}}^\frac12\pr{1+\frac{L}{RS}}^\frac12\pr{1+\frac{H^2CLX^2}{(RS)^4}}^\frac14,
}
and summing over $L$ and integrating over $y$ completes the proof of the proposition.
\end{proof}

\section{The proof of Theorem~\ref{thm:asymptotic}} \label{sec:Thm1}

We start by expressing $\pmd{\zeta\pr{\frac12+it}}^2$  as a sum of length approximately $T^{1+\eps}$. Let $G(w)$ be an entire function with rapid decay along vertical lines, that is $G(x + iy) \ll y^{-A}$ for any fixed $x$ and $A > 0.$ Suppose $G(-w) = G(w), G(0) =1, G(1/2) = 0$. 
We will use the following form of the approximate functional equation for
$|\zeta(s)|^2$. 

\begin{lemma}[Approximate functional equation] \label{lem:fncofZetasq}For $T < t < 2T,$ we have
\est{
\pmd{\zeta\pr{\frac12+it}}^2= 2 \sum_{m_1, m_2}\frac{1}{(m_1m_2)^{\frac12}}\pr{\frac{m_1}{m_2}}^{it}W\pr{\frac{2\pi m_1m_2}t} + O\left(T^{-2/3} \right),
}
where
\est{
W(x):= \frac{1}{2\pi i} \int_{(2)}x^{-w}G(w)\frac{dw}w,
}
and where we use the notation $\int_{(c)}$ to mean an integration up the vertical line from $c-i\infty$ to $c+i\infty$.

\end{lemma}
The proof of the lemma can be found in Lemma 3 of \cite{LR}.
\begin{rem*}
Notice that $W^{(\ell)}(x) \ll_{\ell,A} \min\pr{1,x^{-A}}$ for $x>0$ and all $\ell\in\N$ .
\end{rem*}

The error term in Lemma~\ref{lem:fncofZetasq} produces an error term bounded by $T^{\frac 13 + \eps}$, and thus
\est{
I &= 2\sum_{n_1, n_2, m_1, m_2} \frac{a_{n_1}\overline a_{n_2}}{(m_1m_2n_1n_2)^{\frac12}}\int_{\R}\pr{\frac{m_1n_2}{m_2n_1}}^{it}W\pr{\frac{2\pi m_1m_2}t}\phi\pr{\frac tT}\,dt  + O(T^{\frac 13 + \eps})\\
&= \mathcal D + \s + O(T^{\frac 13 + \eps}),
}
where the sum is over $n_1, n_2 \leq N$, $\mathcal D$ is the sum when $m_1n_2 = m_2n_1$, and $\s$ is the sum when $m_1n_2 \neq m_2n_1.$

\subsection{Diagonal terms} \label{sec:diagonalterm}
Firstly, we consider the diagonal terms  $m_1n_2= m_2n_1$. For $j=1,2$, we write $m_j=\ell n_j^*$, where $n_j^* = \frac{n_j}{(n_1,n_2)}$. The contribution of the diagonal term is
\es{\label{DT}
\mathcal D = 2\sum_{n_1,n_2, \ell}&\frac{a_{n_1}\overline a_{n_2}(n_1,n_2)}{\ell n_1n_2}\int_{\R}W\pr{\frac{2\pi \ell^2 n_1^*n_2^*}{t}}\phi\pr{\frac tT}\,dt\\
&= \frac{2}{2\pi i} \sum_{n_1,n_2,\ell}\frac{a_{n_1}\overline a_{n_2}(n_1,n_2)}{\ell n_1n_2}\int_{\R}\int_{(2)}\pr{\frac{2\pi \ell^2 n_1^*n_2^*}{t}}^{-w}G(w)\frac{dw}w\phi\pr{\frac tT}\,dt\\
&=\frac{2}{2\pi i} \sum_{n_1,n_2}\frac{a_{n_1}\overline a_{n_2}(n_1,n_2)}{ n_1n_2}\int_{\R}\int_{(2)}\pr{\frac{t}{2\pi  n_1^*n_2^*}}^{w}\zeta(1+2w)G(w)\frac{dw}w\phi\pr{\frac tT}\,dt.
}

This term will be later combined with a contribution from the off-diagonal terms. Together, they give the main term in Theorem~\ref{thm:asymptotic}.

\subsection{Off-Diagonal terms}
In this section, we consider the terms with $m_1n_2 \neq m_2n_1$.  We write $m_1n_2- m_2n_1 = \Delta.$  

Since $W(x) \ll x^{-A}$ when $x \gg 1,$ we can truncate the sum over $m_1, m_2$ to when $m_1m_2 \leq T^{1 + \eps}.$ We introduce a smooth partition of unity
\es{\label{pofu}
1=\sumprime_{M}F_{M}(x),\qquad  T^{-100}\leq x\leq T^{1+\eps},
}
where $F_{M}(x)$ is smooth, supported in $[M/2,3M]$, and it satisfies $F^{(j)}_{M}(x)\ll_j \frac{1}{M^j}$ for all $j\geq 0$. 
Moreover we can choose a partition of unity which satisfies $\sum'_{M}1\ll \log (2+T)$. Therefore
\es{\label{frS}
\s&= 2 \sumprime_{N_1}\sumprime_{N_2}\sumprime_{M} \sum_{\Delta \neq 0}\sum_{\substack{n_1, n_2, m_1, m_2 \\ m_1n_2 - m_2n_1 = \Delta} } \frac{a_{n_1}\overline a_{n_2}}{(m_1m_2n_1n_2)^{\frac12}} \\
&\quad \times \left(\int_{\R}\pr{1 + \frac{\Delta}{m_2n_1}}^{it}W\pr{\frac{2\pi m_1m_2}t}\phi\pr{\frac tT}\,dt \right) F_{N_1}(n_1) F_{N_2}(n_2)F_M\left( m_2\right)+O(1),
}
where $N_1, N_2 \leq N$ and $M \leq T^{1 + \eps}.$

Next we show that the terms with $|\Delta|> D$, $D:=\frac{MN_1 }{T^{1-\varepsilon}}$, give a negligible contribution. In fact, 
\est{
\frac{d^\ell}{dt^\ell}W\pr{\frac{2\pi m_1m_2}t}\ll_{\ell,A}\frac1{t^\ell}\min\pr{1,\pr{\frac{2\pi m_1m_2}t}^{-A}},
}
whence, integrating by part $\ell$ times, we have 
\est{
&\sumprime_{N_1,N_2,M}\sum_{|\Delta| > D} \sum_{n_1, m_2} \sum_{\substack{n_2, m_1\\ m_1n_2- m_2n_1=\Delta}}\frac{a_{n_1}\overline a_{n_2}}{(m_1m_2n_1n_2)^{\frac12}} \left(\int_{\R}\pr{1+\frac{\Delta}{m_2n_1}}^{it}W\pr{\frac{2\pi m_1m_2}t}\phi\pr{\frac tT}\,dt\right) \\
&\hskip 3in \times F_{N_1}(n_1) F_{N_2}(n_2)F_M\left( m_2\right)\\
&\quad\ll_\ell \sumprime_{N_1,M} \sum_{|\Delta| > D} \sum_{\substack{n_1 \sim N_1, m_2 \sim M}} \sum_{\substack{m_1 \leq T^{1 + \eps} \\ n_2 \leq N,\\  m_1n_2- m_2n_1=\Delta}}\frac{T^{-\ell+1+\varepsilon}}{(m_1m_2n_1n_2)^{\frac12}}\pmd{\log\pr{1+\frac{\Delta}{m_2n_1}}}^{-\ell} \\
&\quad\ll_\ell  \sumprime_{N_1,M}  \sum_{|\Delta| > D}  \sum_{\substack{n_1 \sim N_1, m_2 \sim M}} \frac{1}{\sqrt{n_1m_2}}\sum_{\substack{m_1 \leq T^{1 + \eps} \\ n_2 \leq N,\\  m_1n_2- m_2n_1=\Delta}}\frac{1}{\sqrt{n_2m_1}}T^{-\ell+1+\varepsilon} \cdot \pr{\frac{m_2n_1}{|\Delta|}}^\ell 
 \ll_{A, \varepsilon} T^{-A},
}
where $\ell$ is large enough. 

Now, if $|\Delta|< D$, then $\frac{\Delta}{m_2n_1} \ll \frac{1}{T^{1-\eps}},$
and
\est{
m_1=\frac{m_2n_1+\Delta}{n_2}=m_2\frac{n_1}{n_2}\pr{1+ \frac{\Delta}{m_2n_1}}.
}
 
Hence for $T < t < 2T,$

$$ \frac{1}{m_1} = \frac{n_2}{m_2n_1} \left(  1 - \frac{\Delta}{m_2n_1} +O\pr{\frac1{T^{2-\varepsilon}}}\right);$$

\est{
\pr{1+\frac{\Delta}{m_2n_1}}^{it}=e^{it\log\pr{1+\frac{\Delta}{m_2n_1}}}
=e^{it\frac{\Delta}{m_2n_1}}\pr{1 - \frac{it\Delta^2}{2m_2^2n_1^2} + O\pr{\frac{1}{T^{2-\varepsilon}}}},
}
and
$$ W\left( \frac{2\pi m_1m_2}{t}\right) = W\left( \frac{2\pi m_2^2 n_1}{tn_2}\right) + \frac{2\pi m_2\Delta}{tn_2} W'\left( \frac{2\pi m_2^2n_1}{\new{t}n_2}\right) +O\left(\frac{1}{T^{2-\eps}}\right).$$
Since $m_1m_2 \leq T^{1 + \eps},$ we have $m_2(m_2n_1 + \Delta) \leq n_2 T^{1 + \eps}.$ Hence $M \ll T^{1/2 + \eps} \sqrt{\frac{N_2}{N_1}},$ and the error term from using the above approximations in~\eqref{frS} is 
\begin{align*}
&\ll \frac{T}{T^{2-\eps}} \sumprime_{N_1} \sumprime_{N_2}\sumprime_{M \ll T^{1/2 + \eps} \sqrt{\frac{N_2}{N_1}}} \sum_{0<|\Delta|\leq D} \ \sum_{\substack{n_1 \sim N_1 \\ n_2 \sim N_2}}\sum_{m_2 \sim M}\frac{1}{m_2n_1} \\
&\ll  \sumprime_{N_1} \sumprime_{N_2}\sumprime_{M \ll T^{1/2 + \eps}\sqrt{{N_2}/{N_1}}} \frac{M N_1 N_2}{T^{2 - \varepsilon}} \ll \frac{\sqrt{T} N_2^{3/2} N_1^{1/2}}{T^{2 - \varepsilon}} \ll \frac{N^2}{T^{3/2 - \varepsilon}},
\end{align*}
using that $D = M N_1 / T^{1 - \varepsilon}$ and that
$M \ll T^{1/2 + \varepsilon} \sqrt{\frac{N_2}{N_1}}$. 
Thus, we have 
\est{
\s=\mathcal A + \mathcal{E}+O\pr{1+\frac{N^2}{T^{3/2 - \varepsilon}}},
}
where
\es{\label{fl}
\mathcal A &= 2 \sumprime_{N_1}\sumprime_{N_2} \sumprime_{M \leq T^{1/2 + \eps}\sqrt{\frac{N_2}{N_1}}} \ \sum_{0<|\Delta|\leq D} \sum_{n_1,n_2} \ \sum_{\substack{n_1m_2\equiv -\Delta\mod {n_2} \\ m_2 > 0} } \ \frac{a_{n_1}\overline a_{n_2}}{m_2n_1}\times\\
&\quad\times\left(\int_{\R}\e{\frac{\Delta t}{2\pi m_2n_1}} W\pr{\frac{2\pi m_2^2n_1}{tn_2}}\phi\pr{\frac tT}\,dt \right)F_M(m_2) F_{N_1}(n_1) F_{N_2}(n_2),
}
and
\es{\label{eqn:errorLinear}
\EE &= 2 \sumprime_{N_1}\sumprime_{N_2} \sumprime_{M \leq T^{1/2 + \eps}\sqrt{\frac{N_2}{N_1}}} \ \sum_{0<|\Delta|\leq D} \sum_{n_1,n_2} \ \sum_{\substack{n_1m_2\equiv -\Delta\mod {n_2} \\ m_2 > 0} } \ \frac{a_{n_1}\overline a_{n_2}}{m_2n_1}F_M(m_2) F_{N_1}(n_1) F_{N_2}(n_2)\times\\
&\quad\times \int_{\R}\e{\frac{\Delta t}{2\pi m_2n_1}} \left[W\pr{\frac{2\pi m_2^2n_1}{tn_2}} \pr{-\frac{\Delta}{2m_2n_1} - \frac{it\Delta^2}{2m_2^2n_1^2}} + \frac{2\pi m_2\Delta}{tn_2}W'\pr{\frac{2\pi m_2^2n_1}{tn_2}}\right]\phi\pr{\frac tT}\,dt,
}
since the rest of the terms arising from the above approximations also give a contribution which is $O\pr{N^2T^{-3/2+ \varepsilon}}$.

First, we consider $\mathcal A$. Giving an eligible bound for $\EE$ is easy and we will do it in the next section. 

Extracting the common divisor $d$ of $n_1$ and $n_2$, we re-write the sum~\eqref{fl} as
%
%By changing variables for $n_i$ and $\Delta$, we have that
%
\est{
\mathcal A &=2 \sum_{d \leq N} \frac{1}{d}\ \sumprime_{N_1,N_2 \leq N} \ \sumprime_{M \leq T^{1/2 + \eps}\sqrt{\frac{N_2}{N_1}}} \ \sum_{0<|\Delta|\leq \frac{D}{d}} \sum_{\substack{n_1,n_2 \\ (n_1,n_2) = 1}} a_{dn_1}\overline a_{dn_2}\ F_{N_1}(d n_1) F_{N_2}(d n_2) {\mathcal A}_{M, N_i}(n_1, n_2, \Delta) ,\\
}
where 
$$  \mathcal A_{M, N_i}(n_1, n_2, \Delta) = \sum_{\substack{m_2\equiv -\overline{n_1}\Delta\mod {n_2} } } \ \frac{F_M(m_2)}{m_2n_1} \ \left(\int_{\R}\e{\frac{\Delta t}{2\pi m_2n_1}} W\pr{\frac{2\pi m_2^2n_1}{tn_2}}\phi\pr{\frac tT}\,dt \right).$$
By Poisson summation formula, 
\est{
\mathcal A_{M, N_i}\pr{n_1,n_2, \Delta}&=\frac1{n_1n_2}\sum_{h\in\Z}\e{-\frac{ h \overline n_1\Delta}{n_2}}\int_0^\infty \e{-\frac{hx}{n_2}}\frac{F_{M}\pr{x}}{x}\times\\
&\quad\times\int_\R\e{\frac{\Delta t}{2\pi xn_1}}W\pr{\frac{2\pi x^2n_1}{tn_2}}\phi\pr{\frac tT}\,dt\,dx.
}
After the change of variable $x\rightarrow \frac x{n_1}$, this becomes
\est{
\mathcal A_{M, N_i}\pr{n_1,n_2,\Delta}& = \frac1{n_1n_2}\sum_{h\in\Z}\widetilde {\mathcal A}_{M, N_i}\pr{h,n_1,n_2,\Delta} \e{-\frac{ h \overline n_1\Delta}{n_2}},
}
where
\est{
\widetilde {\mathcal A}_{M, N_i}\pr{h,n_1,n_2,\Delta}&=\int_0^\infty \e{-\frac{hx}{n_1n_2}}\frac{F_{M}\pr{\frac x{n_1}}}{x}\int_{\R}\e{\frac{\Delta t}{2\pi x}}W\pr{\frac{2\pi x^2}{n_1n_2t}}\phi\pr{\frac tT}\,dt\, dx.
}
To understand the contribution of $\widetilde {\mathcal A}_{M, N_i}\pr{h,n_1,n_2,\Delta}$, we consider
the following three cases.
\subsubsection*{Case 1: $h = 0$.}  The contribution to $\mathcal A$ from $h = 0$ is 
\es{\label{asa}
\mathcal A_0 &= 2 \sum_{d \leq N}  \sumprime_{N_1 \leq N}\ \sumprime_{N_2 \leq N} \ \sumprime_{M \leq T^{1/2 + \eps}\sqrt{\frac{N_2}{N_1}}} \ \sum_{0<|\Delta|\leq \frac Dd} \sum_{\substack{n_1,n_2 \\ (n_1,n_2) = 1}} \frac{a_{dn_1}\overline a_{dn_2}\ F_{N_1}(d n_1) F_{N_2}(d n_2)}{dn_1n_2}  \\
& \hskip 1in \times \int_\R\int_0^\infty F_M\left( \frac{x}{n_1}\right)\e{\frac{\Delta t}{2\pi x}}W\pr{\frac{2\pi x^2}{n_1n_2t}}\, \frac{dx}x \phi\pr{\frac tT} \,dt.
}
Now, we can extend the sum over $\Delta$ to $\Delta\in\Z\setminus\{0\}$, since it can be shown as before that the terms $|\Delta|\geq D/d$ give a negligible contribution. Making the change of variables $y= t/x$ and integrating by parts twice we see that the second line of~\eqref{asa} is equal to
\est{
&-\frac1{\Delta^2}\int_\R\int_0^\infty \e{\frac{\Delta y}{2\pi }}\frac{d^2}{dy^2}\pr{F_M\left( \frac{t}{n_1y}\right) W\pr{\frac{2\pi t}{n_1n_2y^2}}\frac1y}\, {dy}\, \phi\pr{\frac tT} \,dt\\
&=-\frac1{\Delta^2}\int_\R\int_{R(t,n_{1})} \e{\frac{\Delta y}{2\pi }}\frac{d^2}{dy^2}\pr{F_M\left( \frac{t}{n_1y}\right) W\pr{\frac{2\pi t}{n_1n_2y^2}}\frac1y}\, {dy}\, \phi\pr{\frac tT} \,dt+O\pr{\frac1{\Delta^2}},\\
}
where $R(t,n_1)=\{y\mid T^{-100}<\frac{t}{n_1y}<T^{1/2 + \eps}\sqrt{\frac{N_2}{N_1}}\}$ and where we estimated trivially the part of the integral over $y$ with $y\in\R_{>0}\setminus R(t,n_1)$, using the properties of $W$ and $F_M$ (and $n_1\ll T$). Thus, summing over $M$ we have
\est{
&\sumprime_M\int_\R\int_0^\infty F_M\left( \frac{x}{n_1}\right)\e{\frac{\Delta t}{2\pi x}}W\pr{\frac{2\pi x^2}{n_1n_2t}}\, \frac{dx}x \phi\pr{\frac tT} \,dt=\\
&\hspace{1cm}=-\frac1{\Delta^2}\int_\R\int_{R(t,n_{1})} \e{\frac{\Delta y}{2\pi }}\frac{d^2}{dy^2}\pr{W\pr{\frac{2\pi t}{n_1n_2y^2}}\frac1y}\, {dy}\, \phi\pr{\frac tT} \,dt+O\pr{\frac{\log (2 + T)}{\Delta^2}}\\
&\hspace{1cm}=-\frac1{\Delta^2}\int_\R\int_{0}^\infty \e{\frac{\Delta y}{2\pi }}\frac{d^2}{dy^2}\pr{W\pr{\frac{2\pi t}{n_1n_2y^2}}\frac1y}\, {dy}\, \phi\pr{\frac tT} \,dt+O\pr{\frac{\log (2 + T)}{\Delta^2}}.
}
Therefore, summing over $N_1,N_2$, we have 
\est{
\mathcal A_0 &= -2 \sum_{d \leq N} \sum_{|\Delta|\neq0} \sum_{\substack{n_1,n_2\leq \frac Nd\\ (n_1,n_2) = 1}} \frac{a_{dn_1}\overline a_{dn_2}}{dn_1n_2\Delta^2}  \int_\R\int_{0}^\infty \e{\frac{\Delta y}{2\pi }}\frac{d^2}{dy^2}\pr{W\pr{\frac{2\pi t}{n_1n_2y^2}}\frac1y}\, {dy}\, \phi\pr{\frac tT} \,dt\\
&\quad+O(T^\eps)\\
&= -2 \sum_{d \leq N}  \sum_{\substack{n_1,n_2\leq \frac Nd\\ (n_1,n_2) = 1}} \frac{a_{dn_1}\overline a_{dn_2}}{dn_1n_2}  \int_\R\int_{0}^\infty \sum_{|\Delta|\neq0}  \frac{1}{\Delta^2}\e{\frac{\Delta y}{2\pi }}\frac{d^2}{dy^2}\pr{W\pr{\frac{2\pi t}{n_1n_2y^2}}\frac1y}\, {dy}\, \phi\pr{\frac tT} \,dt\\
&\quad+O(T^\eps)\\
&=\mathcal A_{0,+}+\mathcal A_{0,-} + O(T^\eps),
}
\new{where we can take the sum over $\Delta$ inside the integrals since they converge absolutely, and after a change of variables,}
\est{
\mathcal A_{0,\pm} &= -2 \sum_{d \leq N} \sum_{\substack{n_1,n_2\leq \frac Nd\\ (n_1,n_2) = 1}} \frac{a_{dn_1}\overline a_{dn_2}}{dn_1n_2}  \int_\R\int_{0}^{\infty} \sum_{\Delta=1}^\infty  (e^{iy} + e^{-iy})\frac{d^2}{dy^2}\pr{W\pr{\frac{2\pi \Delta^2 t}{n_1n_2y^2}}\frac1y}\, {dy}\, \phi\pr{\frac tT} \,dt \\
&= \new{-2  \sum_{\substack{n_1,n_2\leq N}} \frac{a_{n_1}\overline a_{n_2} (n_1, n_2)}{n_1n_2}  \int_\R\int_{0}^{\infty} \sum_{\Delta=1}^\infty 2\cos (y) \frac{d^2}{dy^2}\pr{W\pr{\frac{2\pi \Delta^2 t}{n_1^*n_2^*y^2}}\frac1y}\, {dy}\, \phi\pr{\frac tT} \,dt,}
}
where we recall that $n_i^* = \frac{n_i}{(n_1, n_2)}$ for $i = 1, 2.$ We notice that %Notice that we have $\arg\pr{-\frac{2\pi \Delta^2 t}{n_1^*n_2^*y^2}}=0$ and thus, for $\arg y=\mp\frac\pi 2$, the $-1$ factor has to be interpreted as $e^{\mp\pi i}$. This is important in the first line in the manipulations below. Indeed, for $\arg y=\mp\frac\pi 2$
\est{
\sum_{\Delta=1}^\infty \frac{d^2}{dy^2}\pr{\new{W\pr{\frac{2\pi \Delta^2t}{n_1^*n_2^*y^2}}}\frac1y}&= \frac{1}{2\pi i}\sum_{\Delta=1}^\infty\int_{(2)}\pr{\frac{2\pi\Delta^2 t}{ n_1^*n_2^*}}^{-w}(2w-1)(2w-2)y^{2w-3}G(w)\frac{dw}w\\
&= \frac{1}{2\pi i} \int_{(\frac54)}\zeta(2w)\pr{\frac{2\pi t}{ n_1^*n_2^*}}^{-w}(2w-1)(2w-2)y^{2w-3}G(w)\frac{dw}w.\\
}
  For $0<\Re(s)<1$, we have
\est{
\int_{0}^{\infty}\cos(y)y^{s-1}{dy}=\Gamma(s) \cos\pr{\frac{\pi s}{2}}
}
(see, for example,~\cite{GR}, formula 3.381, 5., page 346), whence we are left with
\est{
%&\frac{-2}{2\pi i} \int_{(\frac54)} \Gamma(w-2) (2w-1)(2w-2) \zeta(2w)\pr{\frac{2\pi t}{ n_1^*n_2^*}}^{-w} G(w)\,\frac{dw}w=\\
%&\hspace{5cm}=
\frac{-2}{2\pi i} \int_{(\frac14)} \new{\Gamma(2w)} \cos\pr{\pi  w}\zeta(2w)\pr{\frac{2\pi t}{ n_1^*n_2^*}}^{-w} G(w)\,\frac{dw}w,
}
where we used the multiplication formula for the gamma function, the identity $\cos(x - \pi) = -\cos (x),$ and we moved the line of integration without encountering any pole, due to the assumption that $G(w)$ vanishes at $w=\frac12$. Thus,
\est{
&\mathcal A_0 +O(T^\eps)\\
&=  \frac{2}{2\pi i} \sum_{\substack{n_1,n_2}} \frac{a_{n_1}\overline a_{n_2}(n_1,n_2)}{n_1n_2}   \int_\R
 \int_{(\frac14)} 2\cos\pr{\pi  w}\new{\Gamma(2w)} \zeta(2w)\pr{\frac{2\pi t}{ n_1^*n_2^*}}^{-w} G(w)\,\frac{dw}w
\, \phi\pr{\frac tT} \,dt\\
 &=   \frac{-2}{2\pi i}\sum_{\substack{n_1,n_2}} \frac{a_{dn_1}\overline a_{n_2}(n_1,n_2)}{n_1n_2}   \int_\R
\int_{(-\frac14)}  \zeta(1+2w)\pr{\frac{t}{2\pi n_1^*n_2^*}}^{w} G(w)\,\frac{dw}w
\, \phi\pr{\frac tT} \,dt,\\
}
\new{where we used the functional equation of the Riemann zeta function (e.g. Chapter 10 in \cite{Dav}), and then we made the change of variables $w\rightarrow-w$   and use the fact that $G(w) = G(-w)$.} 

From~\eqref{DT}, we have
\est{ 
\mathcal D + \mathcal A_0 &= \frac{2}{2\pi i} \sum_{n_1, n_2} \frac{a_{n_1}\overline a_{n_2}(n_1, n_2)}{n_1n_2}\int_{\mathbb R} \int_{(2)} \pr{\frac{t}{2\pi  n_1^*n_2^*}}^{w}\zeta(1+2w)G(w)\frac{dw}w\phi\pr{\frac tT}\,dt \\
&\quad-\frac{2}{2\pi i} \sum_{n_1, n_2} \frac{a_{n_1}\overline a_{n_2}(n_1, n_2)}{n_1n_2}\int_{\mathbb R} \int_{(-1/4)} \pr{\frac{t}{2\pi  n_1^*n_2^*}}^{w}\zeta(1+2w)G(w)\frac{dw}w\phi\pr{\frac tT}\,dt +O(T^\eps)\\
&= \sum_{n_1, n_2} \frac{a_{n_1}\overline a_{n_2}(n_1, n_2)}{n_1n_2} \int_{\mathbb R}\left(\log \frac{t}{2\pi n_1^* n_2^*} + 2\gamma\right)\phi\left( \frac{t}{T}\right) \> dt+O(T^\eps),\\
}  
since
$$  {\rm Res}_{w = 0} x^w\zeta(1 + 2w) \frac{G(w)}{w}   = \frac{1}{2}\log x + \gamma.$$
Now we have the main term. The rest of the off-diagonal terms contribute to the error term as shown in the following two cases. 

\subsubsection*{Case 2: $|h| \geq \frac{N_2}{dM}T^{\eps}$}
In this case and in Case 3, we define $H_d := \frac{N_2}{dM}T^{\eps}.$
By changing variable $t = xy$, we have
\est{
\frac{1}{n_1n_2}\widetilde {\mathcal A}_{M,N_i}\pr{h,n_1,n_2,\Delta}&= \frac{1}{n_1n_2}\int_\R \e{\frac{\Delta y}{2\pi }} \int_0^\infty \e{-\frac{hx}{n_1n_2}}F_{M}\pr{\frac{x}{n_1}}W\pr{\frac{2\pi x}{n_1n_2y}}\phi\pr{\frac {yx}T}\, dx \> dy.
}
Since $F_M$ is supported in $[M/2, 3M]$, $x \asymp \frac{N_1M}{d}$. Moreover, $\frac yT \asymp \frac{1}{x} \asymp \frac{d}{N_1M}$ due to the support of $\phi$, and $\frac{1}{n_1n_2y} \ll \frac{T^{\eps_1}}{x} \asymp \frac{dT^{\eps_1}}{N_1M}$ because of the rapid decay of $W.$ Hence integrating by parts $\ell + 1$ times, for $T\leq t\leq 2T$ we have
\est{
&\frac{1}{n_1n_2}\int_0^\infty \e{-\frac{hx}{n_1n_2}}F_{M}\pr{\frac{x}{n_1}}W\pr{\frac{2\pi x}{n_1n_2y}}\phi\pr{\frac {yx}T}\, dx\\
&\ll_{\ell, \eps} \frac{d^2}{N_1N_2} \left( \frac{n_1n_2}{h} \frac{dT^{\eps_1}}{MN_1}\right)^{\ell + 1} \frac{MN_1}{d}\\
&\ll \left(\frac{T^{\eps_1}}{h}\right)^{\ell+1} \left( \frac{N_2}{ dM}\right)^{\ell}.
}
Therefore, the contribution to $\s$ when $|h| > H_d$ is
\est{
&\ll \sum_{d \leq N}\  \sumprime_{\substack{N_1,N_2 \leq N,\\M \leq T^{1/2 + \eps}\sqrt{\frac{N_2}{N_1}}}} \ \sum_{0<|\Delta|\leq \frac Dd} \sum_{\substack{n_1,n_2 \\ (n_1,n_2) = 1}} \frac{a_{dn_1}\overline a_{dn_2}\ F_{N_1}(d n_1) F_{N_2}(d n_2)}{d} \sum_{|h| \geq H_d} \frac{dT}{hN_1M} \left( \frac{N_2T^{\eps_1}}{ dM h}\right)^{\ell}  \\
&\ll T^{-A}, }
when $\ell$ is sufficiently large. Thus, the terms $|h|> H_d$  give a negligible contribution. 

\subsubsection*{Case 3: $0<|h| < H_d$} It is sufficient to consider  the terms $0<h<H_d$. By changing variables $t = yx,$ and $x$ to $xn_1n_2$, we will consider the dyadic contribution 
\est{{\mathcal A}^*_{M, N_1, N_2} &:=  \sum_{\substack{n_1,n_2  \\ (n_1, n_2) = 1}} \sum_{0<|\Delta|\leq \frac{D}{d}} \sum_{0<h < H_d} \frac{a_{dn_1} \overline a_{dn_2}\ F_{N_1}(d n_1) F_{N_2}(d n_2)}{d} \times \\
& \quad\times  \e{-\frac{ h \Delta \overline n_1}{n_2}}\int_0^\infty \e{-hx}F_{M}\pr{x n_2}\int_\R\e{\frac{\Delta y}{2\pi }}W\pr{\frac{2\pi x}{y}}\phi\pr{\frac {yxn_1n_2}T}\,dy\, dx.}

We write $\phi$ in term of its Mellin transform $\widetilde{\phi}$, to separate the variables $n_1$ and $n_2$. Let $h\Delta = a$, $A = \frac{DH_d}d = \frac{N_1N_2}{d^2T^{1 - \eps}},$ and $\nu_{x,y}(a) = \sum_{h\Delta = a} \e{-hx + \frac{\Delta y}{2\pi}} $. Therefore we have
\est{
{\mathcal A}^*_{M, N_1, N_2} &=  \frac{1}{(2\pi i)d}\int_0^\infty\int_\R\int_{(\varepsilon)} W\pr{\frac{2\pi x}{y}}\sum_{0<|a| < A} \nu_{x,y}(a)\times  \\
& \quad\times \sum_{\substack{n_1,n_2  \\ (n_1, n_2) = 1}} \frac{a_{dn_1}\overline a_{dn_2}\ F_{N_1}(d n_1) F_{N_2}(d n_2)F_{M}\pr{xn_2}}{n_1^wn_2^w}  \e{-\frac{a \overline n_1}{n_2}} \widetilde{\phi}(w)\frac{T^{w}}{x^wy^w} \> dw \,dy\, dx.
}
Since $F_M$ is supported in $[M/2, 3M]$, $x \asymp \frac {dM}{N_2}$. Moreover, $y \asymp \frac{T}{xn_1n_2} \asymp \frac{Td}{MN_1}$ because $\phi$ is supported in [1,2]. Thus, using Proposition~\ref{prop:DFIimproved}, we have
\es{ \label{eqn:hSmallbound}
{\mathcal A}^*_{M, N_1, N_2} &\ll  \frac{1}{d}\int_{x \asymp \frac {dM}{N_2}} \int_{y \asymp \frac{Td}{MN_1}}
 \left(\frac{ (N_1N_2A) ^{\frac{17}{20}+\eps}}{d^{\frac{17}{10} - \varepsilon}} \frac{(N_1 + N_2)^{\frac{1}{4}}}{d^{\frac{1}{4}}} +\frac{ (N_1N_2A)^{\frac{7}8+\eps} }{d^{\frac{7}4 - \varepsilon}} \frac{(AN_1 + AN_2)^{\frac{1}{8}}}{d^{\frac{1}{8}}} \right)\,dy\, dx   \\
%&\ll  \frac{Td}{N_1N_2}
 %\left(\frac{ (N_1N_2) ^{\frac{17}{10}+\eps}}{T^{\frac{17}{20}}d^{\frac{17}{5} - \varepsilon}} \frac{(N_1 + N_2)^{\frac{1}{4}}}{d^{\frac{1}{4}}} +\frac{ (N_1N_2)^{\frac{7}4+\eps} }{T^{\frac78}d^{\frac{7}2 - \varepsilon}}\frac{(N_1N_2)^{\frac18}}{T^{\frac18}d^{\frac14}} \frac{(N_1 + N_2)^{\frac{1}{8}}}{d^{\frac{1}{8}}} \right)  \nonumber \\
&\ll  
 \left(T^{\frac{3}{20}+\eps} \frac{ (N_1N_2) ^{\frac{7}{10}}}{d^{\frac{53}{20} }} (N_1 + N_2)^{\frac{1}{4}} +T^\eps\frac{ (N_1N_2)^{\frac{7}8} }{d^{\frac{23}8 }}(N_1 + N_2)^{\frac{1}{8}} \right)   
}
Summing over dyadic intervals for $M \leq T^{1/2+\eps}\sqrt{\frac{N_2}{N_1}}$, $N_i \leq N$, and $d \leq N,$ we have that
the contribution to $\mathcal A$ when $|h| \leq H_d$ is bounded by
$ T^{\frac{3}{20} + \eps} N^{\frac{33}{20}} + T^{ \eps}N^{\frac{15}{8}}.$ Therefore we take $N$ up to $T^{\frac{17}{33}-\eps}$ to obtain an eligible error term in Theorem~\ref{thm:asymptotic}.

\subsection{A trivial bound for $\EE$}
Extracting the common divisor $d$ from $n_1$ and $n_2$, applying Poisson summation formula, and changing variables, we can write~\eqref{eqn:errorLinear} as
\est{
\EE &=  2\sum_{d \leq N} \frac{1}{d} \sumprime_{N_1,N_2 \leq N} \ \sumprime_{M \leq T^{1/2 + \eps}\sqrt{\frac{N_2}{N_1}}}  \sum_{0<|\Delta|\leq \frac{D}{d}}  \sum_{\substack{n_1,n_2 \\ (n_1,n_2) = 1}} a_{dn_1}\overline a_{dn_2}\ F_{N_1}(d n_1) F_{N_2}(d n_2) \EE_{M, N_i}(n_1, n_2, \Delta) ,\\
}
where 
\est{
\EE_{ M, N_i}\pr{n_1,n_2,\Delta}& = \frac1{n_1n_2}\sum_{h\in\Z} \e{-\frac{ h \overline n_1\Delta}{n_2}}\int_0^\infty \e{-\frac{hx}{n_1n_2}}F_{M}\pr{\frac x{n_1}}\int_{\R}\e{\frac{\Delta t}{2\pi x}}\phi\pr{\frac tT} \times\\
&\quad\times \left[W\pr{\frac{2\pi x^2}{n_1n_2t}} \pr{-\frac{\Delta}{2x^2} - \frac{it\Delta^2}{2x^3}} + \frac{2\pi \Delta}{n_1n_2t}W'\pr{\frac{2\pi x^2}{n_1n_2t}}\right] \,dt\, dx.
}
Integrating by parts, as in Case 2 of the previous section, we see that the contribution coming from the terms with $|h| > H_d$ is $O(1)$. Thus, estimating trivially the rest of the terms we have
\est{
\EE_{ M, N_i}\pr{n_1,n_2,\Delta}\ll \frac{T^\eps}{n_1n_2}\pr{1+\frac{N_2}{dM}},
}
whence
\est{
\EE %&\ll  2\sum_{d \leq N} \frac{1}{d} \sumprime_{N_1,N_2 \leq N} \ \sumprime_{M \leq T^{1/2 + \eps}\sqrt{\frac{N_2}{N_1}}}    \sum_{\substack{n_1,n_2 \\ (n_1,n_2) = 1}} \frac{MN_1}{dT}\frac{T^\eps}{n_1n_2}\pr{1+\frac{N_2}{dM}}\\
%&\ll  2\sum_{d \leq N} \frac{1}{d} \sumprime_{N_1,N_2 \leq N} \ \sumprime_{M \leq T^{1/2 + \eps}\sqrt{\frac{N_2}{N_1}}}    \sum_{\substack{n_1,n_2 \\ (n_1,n_2) = 1}} \frac{MT^\eps}{n_2T}\pr{1+\frac{N_2}{dM}}\\
%&\ll  2\sum_{d \leq N} \frac{1}{d} \sumprime_{N_1,N_2 \leq N} \ \sumprime_{M \leq T^{1/2 + \eps}\sqrt{\frac{N_2}{N_1}}}     \frac{M}{d}\frac{T^\eps}{TN_2}\pr{N_2N_1+\frac{N_1N_2^2}{dM}}\\
&\ll  T^{-1/2 + \eps}N +T^{-1+\eps}{N^2}\ll NT^\eps\\
}
and the proof of Theorem~\ref{thm:asymptotic} is complete.

\subsection{The proof of Theorem~\ref{thm:asymwithgeneraltrilinear}}
The proof of Theorem~\ref{thm:asymwithgeneraltrilinear} is the same as Theorem~\ref{thm:asymptotic} except that we use~\eqref{eqn:gentrilinear} instead of Proposition~\ref{prop:DFIimproved} in~\eqref{eqn:hSmallbound}. Notice that~\eqref{eqn:gentrilinear} is applicable, since $A=\frac{N_1N_2}{d^2T^{1 - \eps}}\leq \pr{\frac{N_1N_2}{d^2}}^{\frac{0.5-r}{1+2t}+\eps}$ by $\frac{N_i}d\leq N\leq T^{\frac12+\frac{0.5-r}{1+2(r+2t)}}$. Thus, we obtain that
\begin{align*} 
{\mathcal A}^*_{M, N_1, N_2} &\ll  \frac{T^\eps}{d} \int_{x \asymp \frac {dM}{N_2}} \int_{y \asymp \frac{Td}{MN_1}}
\frac{ \sqrt{N_1N_2A} }{d} \left(\frac{(N_1 + N_2)^{\frac 12 + r}A^{t}}{d^{\frac 12 + r}}\right) + \frac{A N_1}{d} + \frac{AN_2}{d} \,dy\, dx  \nonumber \\
&\ll  T^{\frac12+\eps-t} \left(\frac{(N_1 + N_2)^{\frac 12 + r}(N_1N_2)^{t}}{d^{\frac 32 + r+2t}} \right) + \frac{N_1}{d^2} + \frac{N_2}{d^2}. \nonumber 
\end{align*}
Summing over dyadic intervals for $M, N_i$, we have that the contribution to $\mathcal A$ from these terms is $T^{\frac12+\eps-t} N^{\frac 12 + r+2t}+N^{1+\eps}$ and Theorem~\ref{thm:asymwithgeneraltrilinear} follows.

\section{Proof of Theorem~\ref{thm:zetatimestwopoly}} 
The proof of Theorem~\ref{thm:zetatimestwopoly} follows the proof of Theorem~\ref{thm:asymptotic} except the last part when $0<|h| < H_d$. Here we only modify the last part of the proof using the same arguments by Deshouillers and Iwaniec in \cite{DI}. By the same change of variables, we have to consider 
\est{
{\mathcal A}^*_{M, N_1, N_2} &= \sum_{\substack{n_1,n_2 \\  (n_1, n_2) = 1}}\sum_{0<|\Delta|\leq \frac Dd}  \sum_{0 < |h| < H_d} a_{dn_1}\overline a_{dn_2} F_{N_1}(dn_1) F_{N_2}(dn_2) \times \\
& \times  \e{-\frac{h \overline n_1\Delta}{n_2}}\int_0^\infty \e{-hx}F_{M}\pr{xn_2}\int_\R\e{\frac{\Delta y}{2\pi}}W\pr{\frac{ 2\pi x}{y}}\phi\pr{\frac { xyn_1n_2}T}\,dy\, dx. 
} 
We now write $a_{dn_1}$ as $\alpha_{\mu j}\beta_{\nu r}$, where $\mu | d^{\infty}$, $(d, j) = 1$, $n_1 = \varrho r j,$ $\nu = \frac{d}{(\mu, d)}$ and $\varrho = \frac{\mu}{(\mu, d)}.$ Therefore, we have to bound
\begin{align} \label{eqn:boundtwopolybeforelemmaDI}
&\sum_{d \leq T} \frac{1}{d} \sumprime_{N_1,N_2,M}\sum_{\substack{\mu | d^{\infty} \\ \nu = d/(\mu,d)}} \,  \sum_{0<|\Delta|\leq \frac Dd}  \sum_{0 < |h| < H_d} \sum_{(n_2, \varrho) = 1}  \overline a_{dn_2}  \sum_{(j, dn_2) = 1} \alpha_{\mu j} \sum_{(r, n_2) = 1} \beta_{\nu r}\e{-\frac{h \Delta \overline {\varrho r j}}{n_2}} \nonumber \\
&\times F_{N_1}\left( d\varrho rj \right)F_{N_2}(dn_2) \int_0^\infty \e{-hx}F_{M}\pr{x n_2}\int_\R\e{\frac{\Delta y}{2\pi}}W\pr{\frac{2\pi x}{y}}\phi\pr{\frac {xy\varrho r j n_2}{T}}\,dy\, dx,  
\end{align}
where the sums over $N_1$, $N_2$, $M$ are dyadic sums up to $NK$, $NK$, and $T^{1/2 + \epsilon} \sqrt{N_2/N_1}.$

To bound the above sum, we use Proposition~\ref{lem:DI}. However, first we need to apply Mellin's transform to $F_{N_1}$ and $\phi$ to separate variables $n_2, r, j$. The technique is standard, so we skip the details. From Proposition~\ref{lem:DI}, the sum over $\ell$ is the sum over $n_2$, and $L = \frac{N_2}{d}.$  The sum over $j$ is the sum over $j$, and $J \leq \frac{N}{\mu}$. The sum over $u$ is the sum over $h\Delta$, and $U = \frac{N_1N_2}{d^2T^{1 - 2\eps}}$. Finally the sum over $v$ is the sum over $r$, and $V \leq \frac{K}{\nu}.$ Moreover, we note that $JV\leq \frac{N_1}{d\varrho}.$ Applying Proposition~\ref{lem:DI}, we obtain that~\eqref{eqn:boundtwopolybeforelemmaDI} is bounded by (after summing over dyadic $M$)
\est{&\ll T^{\eps} \sum_{d \leq T} \sum_{\substack{\mu | d^{\infty} \\ \nu = d/(\mu,d)}} \, \sumprime_{N_1}\sumprime_{N_2}\frac{dT}{N_1N_2}  \left(\frac{N_1N_2}{\varrho^{1/2}d^2 T^{1/2}}\right)\left\{ \left( \frac{N_2N}{d\mu}\right)^{1/2} + \left( \frac{N_1N_2}{d^2T} + \frac{K}{\nu}\right)^{1/4}\right. \times \\
&\hskip 1.7in \times \left. \left[ \frac{N_2N}{d\mu} \left( \frac{N_1N_2}{Td^2} + \frac{\varrho K}{\nu}\right)\left(\frac{N_2}{d} + \frac{\varrho K^2}{\nu^2} \right) + \frac{\varrho N_1N_2N^2K^2}{d^2T \mu^2\nu^2} \right]^{1/4}\right\} \\
&\ll \sum_{d \leq T} \sum_{\substack{\mu | d^{\infty} \\ \nu = d/(\mu,d)}} \, \frac{T^{\frac12+\eps}}{\varrho^{1/2}d }\left\{ \left( \frac{N^2K}{d\mu}\right)^{1/2} + \left( \frac{N^2K^2}{d^2T} + \frac{K}{\nu}\right)^{1/4}\right. \times \\
&\hskip 1.7in \times \left. \left[ \frac{N^2K}{d\mu} \left( \frac{N^2K^2}{Td^2} + \frac{\varrho K}{\nu}\right)\left(\frac{NK}{d} + \frac{\varrho K^2}{\nu^2} \right) + \frac{\varrho N^4K^4}{d^2T \mu^2\nu^2} \right]^{1/4}\right\} \\
&\ll T^{\eps} \left( T^{1/2} N^{3/4}K + T^{1/2} NK^{1/2} + N^{7/4}K^{3/2}\right)\sum_{d \leq T} \sum_{\substack{\mu | d^{\infty} \\ \nu = d/(\mu,d) }} \,  \frac{1}{d^{5/4}} \frac{1}{\mu^{1/4}} \\ 
&\ll T^{\eps} \left( T^{1/2} N^{3/4}K + T^{1/2} NK^{1/2} + N^{7/4}K^{3/2}\right),}
and this completes the proof of Theorem~\ref{thm:zetatimestwopoly}.

\section{Proof of Theorem~\ref{thm:zetatimesproductoftwosmooth}}

The proof of Theorem~\ref{thm:zetatimesproductoftwosmooth} follows the proof of Theorem~\ref{thm:asymptotic} except the last part when $0<|h| < H_d$. 

We recall that we have $a_b = \sum_{nk = b} \alpha_n\beta_k$, and we assume that $\alpha_n$ is supported on $[N T^{-\xi_1},2N],$ where $N\ll T^{\frac12+\eps}$ and $0\leq\xi_1\leq \frac{1}{5}$. Moreover $\beta_k$ is supported on $k\in [KT^{-\xi_2},2K]$, and  $0\leq \xi_1\leq\frac1{16}$. Let $\xi_1+\xi_2=\xi$. We introduce smooth partitions of unity in the sums over $n$ and $k$ (without indicating it, to save notation). Thus, $n_i\approx N_i$, $k_i\approx K_i$  (note that in the notation of Section~\ref{sec:Thm1} $N_i$ was the size of $b=nk$, so $N_iK_i$ in the current notation). Thus, $db_i \asymp N_iK_i$. Moreover, we assume that $\alpha_n=\psi(n)$, where $\psi(x)$ is a smooth function such that $\psi^{(j)}(x)\ll_j x^{-j}$ for $j\geq0$.

We have to bound
\est{
{\mathcal A}'_{N,K} &:= \sum_{d \leq NK} \ \sumprime_{N_1, N_2, K_1, K_2} \sumprime_{ M\leq T^{\frac 12 +\varepsilon}\sqrt{\frac{N_2K_2}{N_1K_1}}}  \sum_{\substack{b_1,b_2 \\ (b_1, b_2) = 1}} \  \sum_{0<|\Delta|\leq \frac{MN_1K_1}{dT^{1 -\eps}}} \  \sum_{0< |h| < \frac{N_2K_2T^\eps}{dM}} \\
&\quad \frac{a_{db_1} \overline{a_{db_2}}F_{N_1K_1}(db_1)F_{N_2K_2}(db_2)}{db_1b_2}   \e{-\frac{h \overline b_1\Delta}{b_2}}\int_0^\infty \e{-\frac{hx}{b_1b_2}}\frac{F_{M}\pr{\frac{x}{b_1}}}{x} \times \\
& \quad\times \int_\R\e{\frac{\Delta t}{2\pi x}}W\pr{\frac{2\pi x^2}{b_1b_2t}}\phi\pr{\frac tT}\,dt\, dx.\\
}
After the change of variables $y = \frac{ t}{2\pi x}$ and then $z = \frac {x}{b_1b_2}$, it becomes
\est{
 {\mathcal A}'_{N,K} &= 2\pi\sum_{d \leq NK} \ \sumprime_{N_1, N_2, K_1, K_2} \sumprime_{ M\leq T^{\frac 12 +\varepsilon}\sqrt{\frac{N_2K_2}{N_1K_1}}}  \sum_{\substack{b_1,b_2 \\ (b_1, b_2) = 1}} \  \sum_{0<|\Delta|\leq \frac{MN_1K_1}{dT^{1 -\eps}}} \  \sum_{0< |h| < \frac{N_2K_2T^\eps}{dM}} \\
&\quad \frac{a_{db_1} \overline{a_{db_2}}F_{N_1K_1}(db_1)F_{N_2K_2}(db_2)}{d}  \e{-\frac{h \overline b_1\Delta}{b_2}}\int_0^\infty \e{-hz}F_{M}\pr{zb_2} \times \\
&\quad \times \int_\R\e{\Delta y}W\pr{\frac{z}{y}}\phi\pr{\frac {2\pi yzb_1b_2}{T}}\,dy\, dz,
} 

Firstly we claim that we can truncate the sum over $d$ at height $Y:= (N_1K_1N_2K_2)^{\frac 12}/T^{\frac12+\eta}$ for some small $\eta>0$, up to an error term with a power saving.  This is because for larger values of $d$ we are essentially left with the contribution coming from a Dirichlet polynomial of length $T^{\frac12+\eta}$, which we can bound using the method used to prove Theorem~\ref{thm:asymptotic}. More precisely, by~\eqref{eqn:hSmallbound}, we have that the contribution from the terms with $d\geq Y$ is bounded by
\es{\label{larged}
&\ll \sumprime_{N_1, N_2, K_1, K_2}\sum_{Y \leq d \leq NK} \bigg(\frac{T^{\frac{3}{20} + \eps}(N_1K_1 + N_2K_2)^{\frac {1}{4}}(N_1K_1N_2K_2)^{\frac{7}{10}}}{d^{ \frac{53}{20}}}  +\\
&\hskip 2in +T^\eps{}  \frac{(N_1K_1 + N_2K_2)^{\frac {1}{8}}(N_1K_1N_2K_2)^{\frac{7}{8}}}{d^{ \frac{23}{8}}}\bigg)\\
&\ll \sumprime_{N_1, N_2, K_1, K_2} \bigg(T^{\frac{39}{40} +\frac{33}{20}\eta+\eps}\pr{\frac{N_1K_1}{N_2K_2} +\frac{ N_2K_2}{N_1K_1}}^{\frac {1}{8}}  +T^{\frac{15}{16}+\frac{15}{8}\eta+\eps}  \pr{\frac{N_1K_1}{N_2K_2} +\frac{ N_2K_2}{N_1K_1}}^{\frac {1}{16}}\bigg)\\
&\ll T^{\frac{39}{40} +\frac{33}{20}\eta+\frac18\xi+\eps} +T^{\frac{15}{16}+\frac{15}{8}\eta+\frac1{16}\xi+\eps},
}
since $T^{-\xi}\ll \frac{N_1K_1}{N_2K_2} \ll T^{\xi}$.

For the remaining part of the proof we use  Watt's arguments in \cite{Watt}. We write 
\est{
a_{db_i}=\sum_{h_ik_i=db_i}\alpha_{h_i}\beta_{k_i}=\sum_{\substack{f_ig_i=b_i,\ \mu_i\nu_i=d,\\ (g_i,\frac d{\nu_i})=1}}\alpha_{f_i\mu_i}\beta_{g_i\nu_i}, 
}
so that $f_i \asymp N_i/\mu_i$ and $g_i \asymp K_i/\nu_i$. We will apply Proposition~\ref{Watt} to bound
\begin{align}\label{fpreW}
& \sumprime_{M\leq T^{ \frac 12 +\varepsilon}\sqrt{\frac{N_2K_2}{N_1K_1}}}\ \sum_{d \leq Y}\frac1d\sum_{\substack{\mu_1,\nu_1,\mu_2,\nu_2,\\\mu_1\nu_1=\mu_2\nu_2=d}}
\int_{z \asymp \frac{dM}{N_2K_2}}
\int_{  y\asymp \frac{dT}{MN_1K_1}} \left| W\pr{\frac{z}{y}}\times \right. \\
&  \times \sum_{\substack{g_1,g_2\\ \pr{g_i,\frac d{\nu_i}}=1}}\sum_{\substack{f_1,f_2 \\ (f_1g_1, f_2g_2) = 1}}  \sum_{0<|\Delta|\leq \frac{MN_1K_1}{dT^{1-\eps}}}  \sum_{0 < |h| < \frac{N_2K_2T^\eps}{dM}}\alpha_{f_1\mu_1}\beta_{g_1\nu_1}\overline {\alpha_{f_2\mu_2}\beta_{g_2\nu_2}} 
 \e{\frac{h\Delta \overline {f_1g_1}}{f_2g_2}}  \nonumber \\
 &\times \left. F_{N_1K_1}(f_1g_1\mu_1\nu_1)F_{N_2K_2}(f_2g_2\mu_2\nu_2)\e{-hz}\e{\Delta y}F_{M}\pr{zf_2g_2}\phi\pr{\frac {2\pi yzf_1f_2g_1g_2}{T}} \right|\,dy\, dz. \nonumber
\end{align}

Before using Proposition~\ref{Watt} (with $H = \frac{MN_1K_1}{dT^{1-\eps}}$, $C = \frac{N_2K_2T^\eps}{dM}, r = g_1, s = g_2, v = f_1, p = f_2$ in the proposition respectively), we verify that $X^2=\frac{N_1N_2K_1K_2}{HCd^2}=T^{1-\eps}\gg T^{\eps}$, which is clearly satisfied if $\eps$ is small enough, and that 
\est{
\pr{\frac{K_1K_2}{\nu_1\nu_2}}^2\geq \frac{N_1^2K_1^2N_2K_2M}{d^3T^{2-\eps}},\qquad \pr{\frac{K_1K_2}{\nu_1\nu_2}}^2\geq \frac{N_2K_2\mu_1}{N_1\mu_2\nu_2}T^{\eps} = \frac{K_2N_2\mu_1}{dN_1}T^{\eps} = \frac{K_2N_2}{\nu_1 N_1}T^{\eps}.
}
Since $M\leq T^{\frac12 + \eps} (N_2K_2/N_1K_1)^\frac12$, and $d \leq \frac{(N_1K_1N_2K_2)^{\frac 12}}{T^{\frac 12 + \eta}},$ the first condition is implied by $T^{2-\eps} \geq N_1^2N_2^2T^{\eps- \eta}$, which is true if $2\eps<\eta$. The second condition is equivalent to $K_1^2K_2\geq \nu_1\nu_2^2T^{\eps}\frac{N_2}{ N_1}$. This is true as long as $\eta>\tfrac16\xi_2+\tfrac43\eps$, since $N_i \leq T^{1/2 + \eps}$, $K_2/K_1\ll T^{\xi_2}$, and
$$ \nu_1\nu_2^2\frac{N_2}{ N_1}T^{\eps}\leq d^3\frac{N_2}{ N_1}T^\eps\leq \frac{N_1^{\frac 12}N_2^{\frac 52}K_1^{\frac 32}K_2^{\frac32} T^{\eps-3\eta}}{T^{3/2}}\leq K_1^2K_2\frac{K_2^{\frac12}}{K_1^{\frac12}} T^{4\eps-3\eta} \leq K_1^2K_2 T^{4\eps-3\eta+\frac{\xi_2}2}.$$

Applying Proposition~\ref{Watt} with $\delta^{-1}=\max(zC,yH)+1\ll T^\eps$ and using that $\mu_i,\nu_i\leq d\leq \frac{(N_1K_1N_2K_2)^{\frac 12}}{T^{\frac12+\eta}}$ , we obtain  that \eqref{fpreW} is bounded by 
\begin{align*}
&\sumprime_{M\leq T^{\frac{1}{2}+\varepsilon}\sqrt{\frac{N_2K_2}{N_1K_1}}}\sum_{d\leq Y}\frac{T^\eps}d   \sum_{\substack{\mu_1,\nu_1,\mu_2,\nu_2,\\\mu_1\nu_1=\mu_2\nu_2=d}}
 \frac{dM}{N_2K_2}
\frac{dT}{MN_1K_1}
\frac{N_1N_2K_1K_2}{Td^2}  \frac{K_1}{\nu_1}\pr{\frac{N_1}{\mu_1}+\frac{K_2}{\nu_2}T^\frac12}\times\\
&\quad\times\pr{1+\frac{N_1N_2\nu_1\nu_2}{Td^2}}^\frac12  \pr{1+\frac{dN_2}{\mu_2K_1N_1}  }^\frac12  \pr{1+\frac{MN_1N_2^2\nu_1^3\nu_2^2}{(K_1K_2)^2d T\mu_2^2}}^{  \frac14}\\
&\ll \sumprime_{M\leq T^{\frac{1}{2}+\varepsilon}\sqrt{\frac{N_2K_2}{N_1K_1}}}\sum_{d\leq Y}\frac{T^{\eps } }d   \sum_{\substack{\mu_1,\nu_1,\mu_2,\nu_2,\\\mu_1\nu_1=\mu_2\nu_2=d}}
\frac{K_1}{\nu_1}\pr{\frac{N_1}{\mu_1}+\frac{K_2}{\nu_2}T^\frac12}\pr{1+\frac{dN_2}{\mu_2K_1N_1}  }^\frac12\pr{1+\frac{MN_1N_2^2\nu_1^3\nu_2^2}{(K_1K_2)^2d T\mu_2^2}}^{  \frac14}\\
&\ll\sum_{d\leq Y}\frac{T^{\eps}}d   \sum_{\substack{\mu_1,\nu_1,\mu_2,\nu_2,\\\mu_1\nu_1=\mu_2\nu_2=d}}
  \left( \frac{K_1K_2T^\frac12}{\nu_1\nu_2}+
  \frac{K_1N_1}{d} + \frac{K_1^{\frac 38}N_1^{\frac 98}N_2^{\frac 58}}{\mu_1^{\frac 34}\mu_2K_2^{\frac 38}T^{\frac 18}} + \frac{K_1^{\frac 38}K_2^{\frac 58}T^{\frac 38}N_1^{\frac 18}N_2^{\frac 58}}{\nu_1^\frac14 d^{\frac 34}}+{} \right. \\
&\hskip 2in \left.+ \frac{K_1^{\frac 12}K_2T^{\frac 12}N_2^{\frac 12}}{\nu_1\nu_2^{\frac 12}N_1^{\frac 12}} + \frac{K_1^{\frac 12}N_1^{\frac 12}N_2^{\frac 12}}{d^{\frac 12}\mu_2^{\frac 12}} + \frac{N_1^{\frac 58}N_2^{\frac 98}d^{\frac 12}}{\mu_1^{\frac 34}\mu_2^{\frac 32}K_2^{\frac 38}K_1^{\frac 18}T^{\frac 18}} + \frac{T^{\frac 38}N_2^{\frac 98}K_2^{\frac 58}}{N_1^{\frac 38}K_1^{\frac 18}d^{\frac 14}\nu_1^{\frac 14}\mu_2^{\frac 12}}\right)\\
&\ll K_1K_2T^{\frac12+\eps} + K_1N_1T^{\eps} + \frac{K_1^{\frac 38}N_1^{\frac 98}N_2^{\frac 58}}{K_2^{\frac 38}T^{\frac 18-\eps}} + K_1^{\frac 38}K_2^{\frac 58}T^{\frac 38+\eps}N_1^{\frac 18}N_2^{\frac 58}+ \frac{K_1^{\frac 12}K_2T^{\frac 12+\eps}N_2^{\frac 12}}{N_1^{\frac 12}}+{} \\
& \hskip 0.2in + K_1^{\frac 12}N_1^{\frac 12}N_2^{\frac 12} T^\eps + \frac{N_1^{\frac 78}N_2^{\frac {11}8}K_1^{\frac 18}}{K_2^{\frac 18}T^{\frac 38 + \frac{\eta}{2}-\eps}} + \frac{T^{\frac 38+\eps}N_2^{\frac 98}K_2^{\frac 58}}{N_1^{\frac 38}K_1^{\frac 18}} .
\end{align*}
Summing over $K_i$ and $N_i$, we obtain that~\eqref{fpreW} is bounded by 
$$K^2 T^{\frac 12 + \eps} +  T^{\frac {3}{4} + \frac{3\xi_2}{8} + \eps} + K N^{\frac 34}T^{\frac 38 + \eps} + K^{\frac 32}T^{\frac 12 + \eps + \frac{\xi_1}{2}} + T^{\frac 34 + \frac{\xi_2}{8} -\frac{\eta}{2} + \eps} + K^{\frac 12}N^{\frac 34}T^{\frac 38 + \frac{3\xi_1}{8} + \frac{\xi_2}{8}}.$$
Theorem~\ref{thm:zetatimesproductoftwosmooth} then follows by taking $\eta=\frac{\xi_2}6 +3\eps$ and collecting the error term~\eqref{larged}.

\section{Proof of Corollary~\ref{cor:thirdMoment}}

The proof of Corollary~\ref{cor:thirdMoment} requires the following two lemmas.

\begin{lemma}\label{lemma:G}
Let $G$ be a compactly supported function. If $F = - G'$ for $x > 0$ and $F$ is three times continuously differentiable and compactly supported, then,
  \[ \sum_n \frac{G \left( \frac{\log n}{\log x} \right)}{n^s} = \frac{1}{2
     \pi i} \int_{(c)} \zeta (s + w) \cdot \widehat{F} \left( - \frac{iw \log x}{2
     \pi} \right)  \frac{\mathd w}{w} , \]
  for $c>\max(1-\Re(s),0)$ and $x>1$, where $\widehat F$ denotes the Fourier transform of $F$,
  \est{
  \widehat F(x):=\int_{-\infty}^\infty F(u)e^{-2\pi i ux}\,du.
  }
\end{lemma}

\begin{proof}First of all $\widehat{F}(x)$ is entire because
  $F$ is compactly supported. 
We expand the function $\zeta (s + w)$ into its Dirichlet series and compute
  \begin{equation} \label{first}
    \frac{1}{2 \pi \mathi} \int_{(c)} n^{- w} \cdot \widehat{F} \left( - \frac{iw
     \log x}{2 \pi} \right)  \frac{\mathd w}{w}. 
    \end{equation}
  Notice that
  $$
  \widehat{F}\bigg ( - \frac{i w \log x}{2\pi} \bigg ) 
  = \int_{-\infty}^{\infty} F(u) x^{u w} du.
  $$
  Inserting this representation into~\eqref{first} and inter-changing
  integrals, we obtain
  \begin{eqnarray*}
    &  & \int_{-\infty}^\infty F (u) \cdot \frac{1}{2 \pi i} \int_{(c)} \left(\frac{x^u}{ n}\right)^w \cdot
    \frac{\mathd w}{w} \mathd u= \int_{\frac{\log n}{\log x}}^{\infty} F (u) \mathd u = G
    \left( \frac{\log n}{\log x} \right).
  \end{eqnarray*}
  In order to justify the interchange of the two integrations
  we truncate the integral 
  in~\eqref{first} at
  a large height $X$, committing an error which goes to zero as $X\rightarrow\infty$ (since the Fourier transform $\widehat{F}$ will
  decay sufficiently fast), and interchange. Then, we use 
  a Perron formula with error term in order to compute
  the conditionally convergent Perron integral appearing
  above. Taking 
  the height $X \rightarrow \infty$ returns the desired result, as
  stated.
\end{proof}

\begin{lemma} \label{lem:from12tosigma} Let $A \geq 0$ and let $0<\eta<\frac1{66}$ be fixed constants. Let $v\in\R$ and $x < T^{1/2 + \eta}$. Let $s = \sigma + it$, where $\sigma = \frac{1}{2} + \frac{A}{\log T}$ and $T \leq t \leq 2T.$ Then 
$$ \int_{T}^{2T} |\zeta(s)|^2 \cdot \left | \sum_{n \leq x}
\frac{d_{1/2}(n)}{n^{s + i v}} \right |^2 dt \ll  T(\log T)^{9/4},$$
where $d_{1/2}(n)$ are the coefficients of the Dirichlet series expansion 
\est{
\zeta(s)^{\frac12}=\sum_{n\geq1}\frac{d_{1/2}(n)}{n^s},\qquad\Re(s)>1.
}
\end{lemma}
\begin{proof}
Let 
$$
\Phi_{x,v}(s) := \zeta(s) \cdot \sum_{n \leq x} \frac{d_{1/2}(n)}{n^{s + iv}}
$$
and
$$
f_{t}(s) =  \frac{s - 1}{s - 3} \exp \big ( (s - i t )^2 \big ).
$$
Then, by Gabriel's convexity theorem (see \cite{HB}, Lemma 3)
\begin{align*}
\int_{\mathbb{R}} |\Phi_{x,v}(\sigma + i u) f_t(\sigma + i u)|^2 d u
\leq \bigg ( \int_{\mathbb{R}} |\Phi_{x,v}(\tfrac 12 + iu) f_t(\tfrac 12 + iu)|^2 du \bigg )^{\frac{5/2 - \sigma}{2}}
\\
\times \bigg ( \int_{\mathbb{R}} |\Phi_{x,v}(\tfrac 52 + iu) f_t(\tfrac 52 + iu)|^2 du \bigg )^{\frac{\sigma - 1/2}{2}}
\end{align*}
We now integrate both sides over $T \leq t \leq 2T$ and use H\"older's inequality
to get 
\begin{align*}
\int_{\mathbb{R}} |\Phi_{x,v}(\sigma + i u)|^2 \widetilde{f}_T(\sigma + i u) d u
\leq \bigg ( \int_{\mathbb{R}} |\Phi_{x,v}(\tfrac 12 + iu)|^2 \widetilde{f}_T(\tfrac 12 + iu) du \bigg )^{\frac{5/2 - \sigma}{2}}
\\
\times \bigg ( \int_{\mathbb{R}} |\Phi_{x,v}(\tfrac 52 + iu)|^2 \widetilde{f}_T(\tfrac 52 + iu) du \bigg )^{\frac{\sigma - 1/2}{2}}
\end{align*}
where
$$
\widetilde{f}_T (\sigma + i u) := \int_{T}^{2T} |f_t(\sigma + i u)|^2 dt.
$$
Clearly $\widetilde{f}_T(\sigma + i u) \asymp 1$ if $T \leq u \leq 2T$. 
In addition $$
\widetilde{f}_T(\sigma + iu) \ll \begin{cases}
1 & \text{ if } T/2 \leq u \leq 3T \\
e^{-|u|} & \text{ otherwise}.
\end{cases}
$$
We also note that $\Phi_{x,v}(s) \ll (1 + |s|)^{1/4 + \varepsilon} \cdot \sqrt{T}$.
Therefore the previous inequality becomes
\est{
\int_{T}^{2T} |\Phi_{x,v}(\sigma + it)|^2 dt
&\ll \bigg ( \int_{T/2}^{3T} |\Phi_{x,v}(\tfrac 12 + it)|^2 dt+O(T) \bigg )^{\frac{5/2 - \sigma}{2}} \\
&\quad \times \bigg ( 
\int_{T/2}^{3T} |\Phi_{x,v}(5/2 + it)|^2 dt +O(T)\bigg )^{\frac{\sigma - 1/2}{2}}. 
}
According to Theorem 1 the first integral on the right-hand side is 
$O\pr{ T (\log T)^{9/4}}$ while the second integral on the
right hand side is $O( T)$. 

\end{proof}

Let $\delta > 0$ be a small positive real number to be chosen later. 
We pick a parameter $\theta$ close to $1$, with $\delta<\theta<1$, and define
\[ \widehat{F} (z) = e^{2 \pi i (\theta - \delta) z} \cdot \left( \frac{e^{2 \pi i (1 -
   \theta) z} - 1}{2 \pi i (1 - \theta) z} \right)^N \]
with some bounded $N > 10$. 
We see that $F$ is compactly supported on 
$[\theta - \delta,\theta - \delta + (1 - \theta)N]$. 
Define for $x > 0$, 
$$
G(x) = 1 - \int_{0}^{x} F(u) du,
$$
and $G(x)=0$ for $x\leq -1$. Moreover, we let $G(x)$ decay smoothly until $0$ on the interval $[-1, 0]$.   
This way $F = -G'$ for $ x > 0$. We notice that $G(x) = 1$ for $0 < x < \theta - \delta$
and that $G(x) = 1 - \widehat{F}(0) = 0$ for $x > \theta - \delta + (1 - \theta) N$. Finally we notice that $G$ is $N$ times differentiable, and consequently
that $\widehat{G}(x) \ll (1 + |x|)^{-N}$. 

Now we make a choice for $\theta$ and $\delta$.
Let $\theta = \log y / \log x$ with
$y = T^{1/2 + 2\delta}$ and $x$ chosen so that $\theta - \delta
 + ( 1 - \theta) N < 1$. We pick $1 - \theta = (\delta / 2)/ (N - 1)$ so 
that $x = y^{1/(1 - (\delta/2) / (N-1))}$. 
Then, we choose $\delta$
small enough but positive so as to ensure that $x < T^{1/2 + 0.01}$. 

Note that
$$
\widehat{F}\bigg ( - \frac{i w \log x}{2\pi} \bigg ) 
= (y x^{-\delta})^w \cdot \bigg ( \frac{(x/y)^w - 1}{w (1 - \theta) \log x} \bigg )^{N} 
$$
Let $s = \sigma + it$ with $t\asymp T$ and $\sigma = \tfrac 12+\frac{A}{\log T}$, with $A>0$. 
Using Lemma~\ref{lemma:G} and shifting contours to $\Re(w) = \tfrac 12 - \sigma$ we get
\begin{eqnarray*}
  \zeta (s) & = & \sum_n \frac{G \left( \frac{\log n}{\log x} \right)}{n^s} +
  \kappa (4N)^N \frac{(y x^{-\delta})^{1 / 2 - \sigma}}{(\delta \log x)^N} \int_{- \infty}^{\infty}
  \frac{\left| \zeta \left( \tfrac{1}{2} + \mathi t + \mathi v \right) \right|
  \mathd v}{\big( \big( \sigma - \tfrac{1}{2} \big)^2 + v^2 \big)^{(N+1) /
  2}} + O(T^{-1})
\end{eqnarray*}
where $O(1/T)$ is the contribution from the pole at $w = 1 - s$
and with $|\kappa| \leq 1$. 
Let $c(m) = \sum_{m = f e, f , e \leq x} d_{1/2}(e)d_{1/2}(f)$. Importantly, notice that
$c(m) = 1$ for $m \leq x$. Since in addition $G(v) = 0$ for $v > 1$ we
get
\begin{align*}
\sum_{n} \frac{G \left ( \frac{\log n}{\log x} \right )}{n^s}
&=\sum_{n} c(n)\frac{G \left ( \frac{\log n}{\log x} \right )}{n^s}\\
%& = \frac{\log x}{2\pi} \int_{\mathbbm{R}} \sum_{n \leq x} \frac{1}{n^{s + iv}} \cdot
%\widehat{G} \bigg ( \frac{v \log x}{2\pi} \bigg ) dv 
%\\
& = \frac{\log x}{2\pi} \int_{\mathbbm{R}} \sum_{n} \frac{c(n)}{n^{s + iv}} \cdot
\widehat{G} \bigg (\frac{v \log x}{2\pi} \bigg ) dv \\
& = \frac{\log x}{2\pi} \int_{\mathbbm{R}} \bigg ( \sum_{n \leq x} \frac{d_{1/2}(n)}{n^{s + i v}}
\bigg )^2 \widehat{G} \bigg (\frac{v \log x}{2\pi} \bigg ) d v. \\
\end{align*}
Combining the above two equations, we have obtained the following inequality
\begin{align*}
|\zeta(s)| & \leq \log x
\int_{\mathbbm{R}} \bigg | \sum_{n \leq x} \frac{d_{1/2}(n)}{n^{s + iv}}
\bigg |^2 \cdot \left | \widehat{G} \bigg (\frac{v \log x}{2\pi} \bigg ) \right | dv
\\ & + (4N)^N \frac{(y x^{-\delta})^{1/2 - \sigma}}{(\delta \log x)^N} \int_{-\infty}^{\infty}
\frac{|\zeta(\tfrac 12 + i t + i v)| d v}{\big ( \big ( \sigma - \tfrac 12 \big )^2 + v^2 \big )^{(N+1)/2} }
+ O(1/T).
\end{align*}
 Therefore we have obtained
\[ \int_T^{2 T} | \zeta (s) |^3 \mathd t \leqslant 
\log x \int_{\mathbbm{R}} \left | \widehat{G} \pr{ \frac{v \log x}{2\pi} } \right | 
 \int_{T}^{2T} |\zeta(s)|^2 \cdot \left | \sum_{n \leq x}
\frac{d_{1/2}(n)}{n^{s + i v}} \right |^2 dt dv
+ \mathcal{E}+O(T^\eps), \]
where
\begin{align} \label{something}
  \mathcal{E} & \leqslant & (4N)^N \cdot \frac{(y x^{-\delta})^{1 / 2 - \sigma}}{(\delta \log x)^N}
  \int_{- \infty}^{\infty} \bigg ( \int_T^{2 T} | \zeta (s) |^2 \cdot \left|
  \zeta \big ( \tfrac{1}{2} + \mathi t + \mathi v \big ) \right| \mathd t
  \bigg ) \cdot \frac{\mathd v}{\big( \big( \sigma - \tfrac{1}{2} \big)^2
  + v^2 \big)^{(N+1) / 2}}.
\end{align}
By H\"older's inequality and the bound $|\zeta(\tfrac 12 + it)| \ll (1 + |t|)^{1/6 + \varepsilon}$, for $|v| < T^{1/100}$ we have
\est{
  \int_T^{2 T} | \zeta (s) |^2 \left| \zeta \left( \tfrac{1}{2} + \mathi t +
  \mathi v \right) \right| \mathd t & \leqslant  \left( \int_T^{2 T} | \zeta
  (s) |^3 \mathd t \right)^{2 / 3} \left( \int_T^{2 T} | \zeta (1 / 2 + \mathi
  t + \mathi v) |^3 \> dt \right)^{1 / 3}\\
  & \leqslant  M_3 (\sigma, T)^{2 / 3} \cdot \pr{M_3 \left(
  \tfrac{1}{2}, T \right)+ O(|v| T^{1/2 + \varepsilon})}^{1 / 3},
}
where
\[ M_3 (\sigma, T) \assign  \int_T^{2 T} | \zeta (\sigma +
   \mathi t ) |^3 \mathd t. \]
By a minor modification of Lemma 4 in Heath-Brown's paper \cite{HB} we have
\[ M_3 \left( \tfrac{1}{2}, T \right) \leqslant C \cdot T^{(3 / 2)
   (\sigma - 1 / 2)}  M_3 (\sigma, T). \]
Therefore,
\begin{align*}
\int_{T}^{2T} |\zeta(s)|^2 |\zeta(\tfrac 12 + it + iv)| \> \new{dt}
& \leq M_3(\sigma, T)^{2/3}
\cdot \pr{CT^{(\new{3/2})(\sigma - \tfrac 12)} M_3(\sigma, T )+ O(|v| T^{1/2 + \varepsilon})}^{1/3}  \\
& \leq CT^{(1/2)(\sigma - \tfrac 12)} \cdot \pr{M_3(\sigma,T)
+ O( T^{1- \varepsilon})}.
\end{align*}
The contribution of $|v| > T^{1/100}$ to~\eqref{something} is negligible, provided that $N$
is chosen to be large enough. 
We conclude that
\[ \mathcal{E} \leqslant \frac{C (4N)^N}{\delta^N}  \left( \frac{T^{1 / 2} x^{\delta}}{y}
   \right)^{\sigma - \tfrac 12} \cdot \frac{M_3 (\sigma, T)}{(\log x (\sigma
- \tfrac 12))^N} +O_{A}(T^{1 - \varepsilon}). \]
Using Lemma~\ref{lem:from12tosigma}, we find that
$$
\log x \int_{\mathbbm{R}} \left | \widehat{G} \bigg ( \frac{v \log x}{2\pi}\bigg ) \right | 
 \int_{T}^{2T} |\zeta(s)|^2 \cdot \left | \sum_{n \leq x}
\frac{d_{1/2}(n)}{n^{s + i v}} \right |^2 dt dv
\ll T ( \log T)^{9/4}. $$
We have obtained the inequality
\[ M_3 (\sigma, T) \ll T (\log T)^{9 / 4}  +O_{A}(T^{1 - \varepsilon})+ \frac{C (4N)^N}{\delta^N}
   \cdot \left( \frac{T^{1 / 2} x^{\delta}}{y} \right)^{\sigma - \tfrac 12}
   \cdot \frac{M_3(\sigma, T)}{((\sigma - \tfrac 12) \log x)^{N}}.
  \]
Recall that $y = T^{1 / 2 + 2\delta}<x < T$. Since $\sigma = \tfrac{1}{2} + \frac{A}{\log T}$, the third term on the right-hand side in the above equation is less than
$$
\leq C (8N/A)^N \delta^{-N} e^{-\delta A} M_3(\sigma,T) 
$$
with $C$ an absolute constant. Thus, if $A$ is large enough (but bounded) then the third term on the right-hand side in the
above equation is absorbed into the left-hand side, and we conclude that
\[ M_3 (\sigma, T) \ll T (\log T)^{9 / 4}. \]
Since $M_3(\tfrac 12 , T) \ll T^{(3/2) \cdot (\sigma - \tfrac 12)} M_3(\sigma, T)$ by
Lemma 4 of Heath-Brown \cite{HB} and since $\sigma = \tfrac 12 + \frac{A}{\log T}$, we obtain that
$$
M_3(\tfrac 12, T) \ll T (\log T)^{9/4}.
$$

\subsection{Moments of the form $k = 1 + 1/n$} \label{sec:overnMoment}
Since we do not claim the result for moments with $k = 1 + 1/n$  
we only sketch the necessary modifications of the previous argument, for the
convenience of the interested reader. 
In order to adapt our  argument above to moments of the form $1 + 1 / n$, it suffices
to prove the inequality
\begin{align*} | \zeta (s) |^{2 / n} & \ll \log T \int_{\sigma_-}^{\sigma_+}
   \int_{\mathbbm{R}} \pmd{ \hat{G} \pr{\frac{v}{2\pi} } }^{2 / n} \cdot \left| \sum_{m \leqslant
   x} \frac{d_{1 / n} (m)}{m^{s + \sigma + \mathi v / \log x}} \right|^2
   \mathd v \mathd \sigma \\ & +
   \frac{(yx^{- \delta})^{2 (1 / 2 - \sigma) / n}}{(\delta \log x)^{2 N / n}}
   \cdot \log T \int_{\sigma_-}^{\sigma_+} \int_{- \infty}^{\infty} \frac{|
   \zeta (\sigma + \mathi t + \mathi v) |^{2 / n} \cdot \mathd v \mathd
   \sigma}{\left( \left( \sigma - \tfrac{1}{2} \right)^2 + v^2 \right)^{(N +
   1) / n}}, \end{align*}
where $\sigma_{-} = \tfrac 12 - \tfrac{1}{\log T}$, $\sigma_{+} = \tfrac 12 + \tfrac{1}{ \log T}$ and
with the implicit constant depending at most on $n, N$ and the same choice of
parameters $\theta, x, y, \delta$. This is sufficient because the previous argument does not depend on some specific quantification of the dependence on $N$. First we note that
\[ \sum_n \frac{G \left( \frac{\log n}{\log x} \right)}{n^s} =
\frac{1}{2\pi}   \int_{\mathbbm{R}} \left( \sum_{m \leqslant x} \frac{d_{1 / n} (m)}{m^{s +
   \mathi v / \log x}} \right)^n \cdot \hat{G} \pr{\frac{v}{2\pi} } \mathd v. \]
Combining this with Lemma~\ref{lemma:G}, and using the same choice of parameters $\theta,
x, y, \delta$ as before, it follows that
\begin{align*} | \zeta (s) | & \ll \int_{\mathbbm{R}} \pmd{ \hat{G} \pr{\frac{v}{2\pi}  } } \cdot \left|
   \sum_{m \leqslant x} \frac{d_{1 / n} (m)}{m^{s + \sigma + \mathi v / \log
   x}} \right|^n \mathd v + 
   \frac{(yx^{- \delta})^{1 / 2 - \sigma}}{(\delta \log x)^N}
   \int_{\mathbbm{R}} \frac{\left| \zeta \left( \tfrac{1}{2} + \mathi
   t + \mathi v \right) \right| \mathd v}{\left( \left( \sigma - \tfrac{1}{2}
   \right)^2 + v^2 \right)^{(N + 1) / 2}} .
\end{align*}
Taking the $2 / n$ power on both sides, it remains to show that
\[ \left( \int_{\mathbbm{R}} \pmd{ \hat{G} \pr{\frac{v}{2\pi} }} \cdot \left| \sum_{m \leqslant x}
   \frac{d_{1 / n} (m)}{m^{s + \mathi v / \log x}} \right|^n \mathd v
   \right)^{\frac2n} \ll \log T \int_{\sigma_-}^{\sigma_+}
   \int_{\mathbbm{R}} \pmd{ \hat{G} \pr{\frac{v}{2\pi} }}^{\frac2n} \cdot \left| \sum_{m \leqslant
   x} \frac{d_{1 / n} (m)}{m^{s + \sigma + \mathi v / \log x}} \right|^2
   \mathd v \mathd \sigma, \]
and that
\begin{equation} \label{another2}
\bigg ( \int_{-\infty}^{\infty}
\frac{|\zeta(\tfrac 12 + i t + i v)| dv}{\big ( \big ( \sigma - \tfrac 12 \big )^2
+ v^2 \big )^{(N+1)/2}} \bigg )^{2/n} \ll \log T 
\int_{\sigma_{-}}^{\sigma_{+}} \int_{-\infty}^{\infty} 
\frac{|\zeta(\sigma + i t + i v)|^{2/n} dv}{\big ( \big ( \sigma - \tfrac 12 \big )^2 + v^2 \big )^{(N+1)/n}}  d \sigma.
\end{equation}
We will only show how to prove the second inequality since the proof of the
first is very similar.
We bound the integral
$$
\int_{-\infty}^{\infty}
\frac{|\zeta(\tfrac 12 + i t + i v)| dv}{\big ( \big ( \sigma - \tfrac 12 \big )^2
+ v^2 \big )^{(N+1)/2}} 
\leq 
(\sigma - \tfrac 12)^{- N - 1}
\sum_{k} \frac{M_k}{(1 + |k|)^{N + 1}},
$$
where $M_k$ is the maximum of $\zeta(\tfrac 12 + it + iv)$ over
the interval $|v - k (\sigma - \tfrac 12)| < (\sigma - \tfrac 12)/2$.
Therefore
\begin{equation} \label{another}
\bigg ( \int_{-\infty}^{\infty}
\frac{|\zeta(\tfrac 12 + i t + i v)| dv}{\big ( \big ( \sigma - \tfrac 12 \big )^2
+ v^2 \big )^{(N+1)/2}} \bigg )^{2/n}
\leq (\sigma - \tfrac 12)^{-2(N + 1)/n}
\sum_{k} \frac{M_k^{2/n}}{(1 + |k|)^{2(N+1)/n}}.
\end{equation}
By sub-harmonicity,
$$
M_{k}^{2/n} \ll \log T \int_{\sigma_{-}}^{\sigma_{+}}
\int_{(k-1) (\sigma - \tfrac 12) }^{(k + 1) (\sigma - \tfrac 12)}
|\zeta(\sigma + it + ix)|^{2/n} \>dx \>d \sigma.
$$
We conclude that
\begin{align*}
 & (\sigma - \tfrac 12)^{-2(N + 1)/n}
\sum_{k} \frac{M_k^{2/n}}{(1 + |k|)^{2(N+1)/n}}
\\ & \ll \log T \int_{\sigma_{-}}^{\sigma_{+}} 
\sum_{k} \int_{(k -1)(\sigma - \tfrac 12)}^{(k+1)(\sigma - \tfrac 12)}
\frac{|\zeta(\sigma + it + i x)|^{2/n} \> dx \> d \sigma}{((\sigma - \tfrac 12)^2 + (k 
(\sigma - \tfrac 12))^2)^{2(N+1)/n}}
\\ &
\ll   \log T 
\int_{\sigma_{-}}^{\sigma_{+}} \int_{-\infty}^{\infty} 
\frac{|\zeta(\sigma + i t + i v)|^{2/n} dv}{\big ( \big ( \sigma - \tfrac 12 \big )^2 + v^2 \big )^{(N+1)/n}}  \> d \sigma. 
\end{align*}
Combining these equations together, we obtain the desired inequality~\eqref{another2}.

\appendix
\section{On Conjecture~\ref{mconj}}
\begin{prop} Let $A,M,N\geq1$ and let $A\ll (MN)^{\frac12+\eps}$. Then
\est{
\max_{\alpha,\beta,\nu} |S_{A,M,N}|  \gg (AMN)^{\frac12-\eps} (M+N)^{\frac 12} +A(M+N)^{1-\eps},
}
for all $\eps>0$, where the maximum is taken over all choices of coefficients $\alpha_m,\beta_n,\nu_a\ll 1$. 
\end{prop}
\begin{proof}
By the reciprocity relation  $\frac {\overline m}n\equiv-\frac {\overline n}m+\frac1{mn}\mod1$ we can assume $M\geq N$. Moreover, we can assume $N,A\gg M^\eps$ for some small $\eps>0$ and $M$ arbitrary large, since otherwise the result is easy.

First, we consider the case $M^{1-\delta}\gg N$ for some $\delta>0$ and we take $\alpha_m=f(m)$ for some smooth function $f:[M,2M]\rightarrow[0,1]$ which is such that $f^{(j)}(x)\ll_j x^{-j}$ for all $j\geq0$ and $\int_\R f(x)=KM,$ for some $K>0$. Also, let $\beta_n=-\gamma_n$, where $\gamma_n$ is the indicator function of the primes congruent to $1\mod 4$ in $[N,2N]$, and let $\nu_a$ be the indicator function of the primes congruent to $3\mod 4$ in $[A,2A]$.

By Poisson summation, we have
\est{
\sum_mf(m)\e{\frac {a\overline m}n}=K\frac MN(c_n(a)+O(M^{-100})),
}
where 
\est{
c_n(a)=\sum_{\substack{b=1,\\(b,n)=1}}^n\e{\frac{ba}{n}}=\mu\pr{\frac{n}{(n,a)}}\frac{\varphi(n)}{\varphi\pr{\frac{n}{(n,a)}}}
} 
is the Ramanujan sum. It follows that
\est{
S_{A,M,N}&=K\frac MN\sum_a\sum_n\beta_n\nu_a (c_n(a)+O(M^{-100}))\\
&=K\frac MN\sum_a\sum_n \gamma_n\nu_a(1 +O(M^{-100}))\gg {(MA)^{1-\eps}}.\\
}

We now prove 
\est{
\max_{\alpha,\beta,\nu} |S_{A,M,N}|  \gg M(AN)^{\frac12-\eps},
}
which then implies the Proposition even in the case $M^{1-\delta}\ll N$ for all $\delta>0$.

By choosing $\alpha_m$ appropriately, we have
\est{
\max_{\alpha,\beta,\nu}|S_{A,M,N}|  \gg \max_{\beta,\nu}\sum_m  F_{m;\beta,\nu},
}
where
\est{
F_{m;\beta,\nu}:=\pmd{ \sum_a  \sum_{(n,m) = 1} \nu_a\beta_n \e{\frac{a\overline{m}}{n}}}.
}
First, notice that we have
\est{
\max_{\beta,\nu} \sum_{m} F_{m;\beta, \nu} \geq \frac{1}{\varphi(q)^2} \sum_{\substack{\chi_1, \chi_2 \mod{q}}}
\sum_m F_{m;\beta(\chi_1),\nu(\chi_2)}
}
with $q$ any prime greater than  $4(A^4 + N^4)$ and where $\beta(\chi_1), \nu(\chi_2)$ denotes sequences defined by $\beta(\chi_1)_n = \chi_1(n)$ and
$\nu(\chi_2)_a = \chi_2(a)$ respectively. 
 Moreover, by H\"older's inequality,
\begin{align*}
\frac{1}{\varphi(q)^2} & \sum_{\substack{\chi_1, \chi_2 \mod{q}}} 
\sum_m F_{m;\beta(\chi_1),\nu(\chi_2)}^2 \leq\Bigg(\frac{1}{\varphi(q)^2} \sum_{\substack{\chi_1,\chi_2 \mod{q}}} \sum_m F_{m;\beta(\chi_1),\nu(\chi_2)} \Bigg)^{\frac23} \times \\
& \times \Bigg(\frac{1}{\varphi(q)^2} \sum_{\substack{\chi_1, \chi_2 \mod{q}}} \sum_m F_{m;\beta(\chi_1),\nu(\chi_2)}^4 \Bigg)^{\frac13}.
\end{align*}
The left hand side is
\est{
\frac{1}{\varphi(q)^2} \sum_{\substack{\chi_1, \chi_2 \mod{q}}} \sum_m F_{m;\beta(\chi_1),\nu(\chi_2)}^2
=\sum_m\sum_{a}\sum_{(n,m)=1}1\gg MAN,
}
and we also have
\est{
\frac{1}{\varphi(q)^2} \sum_{\substack{\chi_1, \chi_2 \mod{q}}} \sum_m F_{m;\beta(\chi_1),\nu(\chi_2)}^4  = \sum_m\sum_{\substack{a_1a_2=a_3a_4}}\sum_{\substack{n_1n_2=n_3n_4,\\ (m,n_1n_2)=1}}1\ll M(AN)^{2+\eps}.
}
Thus,
\est{
\frac{1}{\varphi(q)^2} \sum_{\substack{\chi_1, \chi_2 \mod{q}}} 
F_{m;\beta(\chi_1),\nu(\chi_2)}  \gg M(AN)^{\frac12-\eps}
}
and the proposition follows.
\end{proof}

\end{document}